\documentclass[10pt,a4paper]{article} 
\usepackage{amsthm}
\usepackage[pdftex]{graphicx} 
\usepackage[pdftex]{hyperref} 
\usepackage{amsmath}
\usepackage{amssymb}
\usepackage{mathtools}
\usepackage{ascmac}
\usepackage{mathrsfs}
\usepackage{setspace}
\usepackage{nccmath}
\usepackage{color}
\usepackage{enumerate}
\usepackage[rm]{titlesec}
\usepackage{aliascnt} 
\usepackage{indentfirst}
\usepackage{tikz}
\usetikzlibrary{positioning}
\usetikzlibrary{topaths,calc}
\usepackage{here}
\usetikzlibrary{cd}
\usepackage{tikz-cd}
\usepackage{comment}

\usepackage[english]{babel}  
\addto\extrasenglish{}
\addto\extrasenglish{}

\usepackage[margin=15mm]{geometry} 
\titleformat*{\section}{\center\large\bfseries}
\titleformat*{\subsection}{\center\bfseries}

\hypersetup{colorlinks=true,citecolor=blue,linkcolor=red,urlcolor=cyan}

\numberwithin{equation}{section}

\theoremstyle{plain} 
\newtheorem{thm}{Theorem}[section] 
 
\newaliascnt{lem}{thm} 
\newtheorem{lem}[lem]{Lemma}
\aliascntresetthe{lem}

\newaliascnt{prop}{thm} 
\newtheorem{prop}[prop]{Proposition}
\aliascntresetthe{prop}

\newaliascnt{cor}{thm} 
\newtheorem{cor}[cor]{Corollary}
\aliascntresetthe{cor}

\newaliascnt{conj}{thm} 
\newtheorem{conj}[conj]{Conjecture}
\aliascntresetthe{conj}

\theoremstyle{definition} 
\newaliascnt{defi}{thm} 
\newtheorem{defi}[defi]{Definition}
\aliascntresetthe{defi}

\newaliascnt{exa}{thm} 
\newtheorem{exa}[exa]{Example}
\aliascntresetthe{exa}

\newaliascnt{rem}{thm} 
\newtheorem{rem}[rem]{Remark}
\aliascntresetthe{rem}

\newaliascnt{nota}{thm} 

\aliascntresetthe{nota}

\newaliascnt{obs}{thm} 

\aliascntresetthe{obs}

\newaliascnt{property}{thm} 

\aliascntresetthe{property}

\newaliascnt{ass}{thm} 
\newtheorem{ass}[ass]{Assumption}
\aliascntresetthe{ass}

\newtheorem*{notation}{Notation}
\newtheorem*{acknowledgements}{Acknowledgements}
\newtheorem*{duality}{Duality}
\newtheorem*{curvshape}{Shape of $\tilde{\kappa}(\alpha,h;x,y)$ on $\alpha\in[0,1]$}
\newtheorem*{long-scale}{Long-scale curvature}
\newtheorem*{how}{How to handle $h^\prime(0)$ and $h^\prime(1)$}
\newtheorem*{Appli}{Further applications of $W_h$}

\newcommand{\dis}{\displaystyle} 
\newcommand{\A}{\alpha}
\newcommand{\B}{\beta}

\newcommand{\dex}{\delta_x}
\newcommand{\dey}{\delta_y}
\newcommand{\K}{\kappa}
\newcommand{\la}{\lambda}
\newcommand{\N}{\mathbb{N}}
\newcommand{\PP}{\mathscr{P}}
\newcommand{\R}{\mathbb{R}} 

\newcommand{\Z}{\mathbb{Z}}
\newcommand{\maru}[1]{\raise0.2ex\hbox{\textcircled{\scriptsize{#1}}}} 
\newcommand{\eps}{\varepsilon} 
\newcommand{\uto}{\uparrow}
\newcommand{\dto}{\downarrow}
 
\DeclareMathOperator\supp{supp}

\newcommand{\rwx}{m_x^\alpha} 
\newcommand{\rwy}{m_y^\alpha} 
\newcommand{\wxy}{W_1\big(m_x^\alpha,m_y^\alpha\big)}
\newcommand{\whxy}{W_h\big(m_x^\alpha,m_y^\alpha\big)}
\newcommand{\kaxy}{\kappa(\alpha;x,y)}
\newcommand{\tkaxy}{\tilde{\kappa}(\alpha,h;x,y)}
\newcommand{\kxy}{\kappa(x,y)}
\newcommand{\tkxy}{\tilde{\kappa}(h;x,y)}

\def\:={\coloneqq} 
\def\diam(#1){\mathsf{diam}(#1)}
\def\W(#1){W_1\big(#1\big)}
\def\wh(#1){W_h\big(#1\big)}
\def\01{\{0,1\}}
\def\blue(#1){\textcolor{blue}{#1}}
\def\red(#1){\textcolor{red}{#1}}

\begin{document}

\title{A new transport distance and its associated Ricci curvature \\ of hypergraphs}
\author{Tomoya Akamatsu\thanks{Department of Mathematics, Osaka University, Osaka 560-0043, Japan (\texttt{u149852g@ecs.osaka-u.ac.jp})
}}
\date{\empty}
\maketitle

\begin{abstract}
The coarse Ricci curvature of graphs introduced by Ollivier as well as its modification by Lin--Lu--Yau have been studied from various aspects.
In this paper, we propose a new transport distance appropriate for hypergraphs and study a generalization of Lin--Lu--Yau type curvature of hypergraphs.
As an application, we derive a Bonnet--Myers type estimate for hypergraphs under a lower Ricci curvature bound associated with our transport distance.
We remark that our transport distance is new even for graphs and worthy of further study.
\end{abstract}

\section{Introduction} \label{Intro}
Let $\PP(X)$ denote the set of all Borel probability measures on complete and separable metric space $(X,d)$ and 
$m_x(\cdot) \in\PP(X)$ be a random walk starting from a point $x\in X$ in $(X,d)$.
We call ($X,d$) equipped with a random walk $m=\big(m_x(\cdot)\big)_{x\in X}$ 
from each point a \emph{metric space with a random walk $m$} and denote it as $(X,d,m)$.
Ollivier \cite{Ol} introduced the coarse Ricci curvature for $(X,d,m)$.
The idea of coarse Ricci curvature is to measure the curvature of $(X,d,m)$ by comparing the distance $d(x,y)$ between $x,y\in X$ 
and the distance between the probability measures $m_x,m_y\in\PP(X)$ at $x,y$.
The following $L^1$-Wasserstein distance is used to measure the distance between two probability measures.

\begin{defi}[$L^1$-Wasserstein distance] \label{W_1 dist}
Let $(X,d)$ be a complete and separable metric space.
The \emph{$L^1$-Wasserstein distance $W_1$} on $\PP(X)$ is defined by 
\begin{align} 
W_1(\mu,\nu) \:= \inf_{\pi\in\Pi(\mu,\nu)} \int_{X\times X} d(x,y) \, \textrm{d}\pi(x,y) \label{W1 dist}
\end{align}
for $\mu, \nu\in\PP(X)$, where $\Pi(\mu,\nu)$ denotes the set of all \emph{couplings} of $\mu,\nu$, i.e., 
$\pi\in\Pi(\mu,\nu)$ is a probability measure on $X\times X$ satisfying 
$\pi(A\times X)=\mu(A),\pi(X\times B)=\nu(B)$ for any Borel subsets $A,B\subset X$.
\end{defi}

A coupling that attains the infimum of the right-hand side of \eqref{W1 dist} is called an \emph{optimal coupling} of $(\mu,\nu)$.
An optimal coupling is not necessarily unique.

\begin{defi}[Coarse Ricci curvature for a metric space with a random walk; {\cite[Definition 3]{Ol}}] \label{coarseRicci}
Let $(X,d,m)$ be a metric space with a random walk $m$ and let $x,y\in X$ $(x\neq y)$.
Define
\begin{align*} 
\K_{\mathrm{Ol}}(x,y) \:= 1-\frac{W_1(m_x,m_y)}{d(x,y)}
\end{align*}
to be the \emph{coarse Ricci curvature} of $(X,d,m)$ along $(x,y)$.
\end{defi}

Coarse Ricci curvature can be defined especially in discrete spaces and was modified as Ricci curvature for graphs by Lin--Lu--Yau \cite{LLY} (\autoref{LLY-curv}) 
that we call the LLY--Ricci curvature.
Note that the LLY--Ricci curvature is a quantity defined for edges in graphs.
Lin--Lu--Yau showed a Bonnet--Myers type inequality on the diameter of graphs, and 
a Lichnerowicz type inequality on the first nonzero eigenvalue of the Laplacian for graphs of their Ricci curvature bounded below.

As a generalization of graphs which are representations of binary relations, there are hypergraphs which represent higher dimensional relations 
(see \autoref{hypergraph} for the definition of hypergraphs).
Recently, there is growing interest on optimal transports and curvature of hypergraphs.
Our hypergraphs will be undirected unless otherwise stated.
As for the study of Ricci curvature of hypergraphs, Asoodeh--Gao--Evans \cite{AGE} first introduced curvature of hyperedges of hypergraphs 
and Eidi--Jost \cite{EJ} introduced curvature of hyperedges of directed hypergraphs.
In addition, very recently, Ikeda--Kitabeppu--Takai--Uehara \cite{IKTU} introduced coarse Ricci curvature between two vertices of hypergraphs. 
In \cite{IKTU}, they defined coarse Ricci curvature between two points using the submodular hypergraph Laplacian introduced in \cite{Yo} 
without via a random walk on a hypergraph.
They proved that their curvature coincides with LLY--Ricci curvature for graphs 
and also derived Bonnet--Myers type and Lichnerowicz type inequalities and a gradient estimate for their curvature in \cite{IKTU}.
Moreover, Kitabeppu--Matsumoto \cite{KM} proved a kind of Cheng maximal diameter theorem for this curvature.

In hypergraphs, there is a problem if we simply use the $L^1$-Wasserstein distance to measure the distance between two random walks as in the case of graphs.
That is to say, the transport by $L^1$-Wasserstein distance regards the original hypergraph as a graph with each hyperedge as a clique.
For instance, consider a hypergraph as in \autoref{intro-hypergraph}.
Let $\A$ be sufficiently close to $1$.
The weights on each vertex by random walks $\rwx,\rwy$ (see \autoref{random walk}) are as in \autoref{weight of vertex}.

\begin{figure}[H]
\begin{tabular}{ccc}
\begin{minipage}{0.43\hsize}
\begin{center}
\begin{tikzpicture}
\node (v1) at (0,0) {};
\node (v2) at (1.2,0) {};
\node (v3) at (2.4,0) {};
\node (v4) at (0,1.2) {};
\node (v5) at (1.2,1.2) {};
\node (v6) at (2.4,1.2) {};
\node (v7) at (0,2.4) {};
\node (v8) at (1.2,2.4) {};
\node (v9) at (2.4,2.4) {};
\foreach \v in {1,2,...,9} {\fill (v\v) circle (0.1);} 
\node at ($(v7) + (-0.2,0.2)$) {$y$};
\node at ($(v5) + (0.2,-0.2)$) {$x$};
\path[draw=red] ($(v4) + (-0.5,0)$) 
to[out=90,in=270] ($(v7) + (-0.5,0)$)
to[out=90,in=180] ($(v7) + (0,0.5)$)
to[out=0,in=180] ($(v8) + (0,0.5)$)
to[out=0,in=90] ($(v8) + (0.5,0)$)
to[out=270,in=90] ($(v5) + (0.5,0)$)
to[out=270,in=0] ($(v5) + (0,-0.5)$)
to[out=180,in=0] ($(v4) + (0,-0.5)$)
to[out=180,in=270] ($(v4) + (-0.5,0)$);
\path[draw=green] ($(v5) + (-0.4,0)$) 
to[out=90,in=270] ($(v8) + (-0.4,0)$)
to[out=90,in=180] ($(v8) + (0,0.4)$)
to[out=0,in=180] ($(v9) + (0,0.4)$)
to[out=0,in=90] ($(v9) + (0.4,0)$)
to[out=270,in=90] ($(v6) + (0.4,0)$)
to[out=270,in=0] ($(v6) + (0,-0.4)$)
to[out=180,in=0] ($(v5) + (0,-0.4)$)
to[out=180,in=270] ($(v5) + (-0.4,0)$);
\path[draw=orange] ($(v1) + (-0.4,0)$) 
to[out=90,in=270] ($(v4) + (-0.4,0)$)
to[out=90,in=180] ($(v4) + (0,0.4)$)
to[out=0,in=180] ($(v5) + (0,0.4)$)
to[out=0,in=90] ($(v5) + (0.4,0)$)
to[out=270,in=90] ($(v2) + (0.4,0)$)
to[out=270,in=0] ($(v2) + (0,-0.4)$)
to[out=180,in=0] ($(v1) + (0,-0.4)$)
to[out=180,in=270] ($(v1) + (-0.4,0)$);
\path[draw=cyan] ($(v2) + (-0.5,0)$) 
to[out=90,in=270] ($(v5) + (-0.5,0)$)
to[out=90,in=180] ($(v5) + (0,0.5)$)
to[out=0,in=180] ($(v6) + (0,0.5)$)
to[out=0,in=90] ($(v6) + (0.5,0)$)
to[out=270,in=90] ($(v3) + (0.5,0)$)
to[out=270,in=0] ($(v3) + (0,-0.5)$)
to[out=180,in=0] ($(v2) + (0,-0.5)$)
to[out=180,in=270] ($(v2) + (-0.5,0)$);
\end{tikzpicture}
\caption{A hypergraph $H$: 
$H$ consists of $9$ vertices including $x,y$ and $4$ hyperedges represented by 4 colored curves.} \label{intro-hypergraph}
\end{center}
\end{minipage} 
\begin{minipage}{0.1\hsize}
\begin{center}
\end{center}
\end{minipage}
\begin{minipage}{0.43\hsize}
\begin{center}
\begin{tikzpicture}
\node (v1) at (0,0) {};
\node (v2) at (1.2,0) {};
\node (v3) at (2.4,0) {};
\node (v4) at (0,1.2) {};
\node (v5) at (1.2,1.2) {};
\node (v6) at (2.4,1.2) {};
\node (v7) at (0,2.4) {};
\node (v8) at (1.2,2.4) {};
\node (v9) at (2.4,2.4) {};
\foreach \v in {1,2,...,9} {\fill (v\v) circle (0.1);} 
\node at ($(v7) + (-0.2,0.2)$) {\red($\A$)};
\node at ($(v8) + (-0.2,0.2)$) {\red($\frac{\B}{3}$)};
\node at ($(v4) + (-0.2,0.2)$) {\red($\frac{\B}{3}$)};
\node at ($(v5) + (-0.2,0.2)$) {\red($\frac{\B}{3}$)};
\node at ($(v5) + (0.2,-0.2)$) {\blue($\A$)};
\node at ($(v1) + (0.2,-0.2)$) {\blue($\frac{\B}{8}$)};
\node at ($(v2) + (0.2,-0.2)$) {\blue($\frac{\B}{8}$)};
\node at ($(v3) + (0.2,-0.2)$) {\blue($\frac{\B}{8}$)};
\node at ($(v4) + (0.2,-0.2)$) {\blue($\frac{\B}{8}$)};
\node at ($(v6) + (0.2,-0.2)$) {\blue($\frac{\B}{8}$)};
\node at ($(v7) + (0.2,-0.2)$) {\blue($\frac{\B}{8}$)};
\node at ($(v8) + (0.2,-0.2)$) {\blue($\frac{\B}{8}$)};
\node at ($(v9) + (0.2,-0.2)$) {\blue($\frac{\B}{8}$)};
\end{tikzpicture}
\caption{The weights of each vertex by $\rwx,\rwy$ are drawn in \blue(blue) and \red(red), respectively.
Here, $\B\:=1-\A$.} \label{weight of vertex}
\end{center}
\end{minipage}
\end{tabular}
\end{figure}

Then, in terms of the $L^1$-Wasserstein distance, each hyperedge is considered to be a complete subgraph so that 
one cannot see the difference between the hypergraph $H$ and its clique expansion $H^\prime$ as shown in \autoref{clique expansion}.
In this paper, to distinguish $H$ and $H^\prime$, we introduce a new transport method based on the structure of hyperedges, 
i.e., a new distance between probability measures.
Roughly speaking, when we transport from $\rwx$ to $\rwy$ in \autoref{weight of vertex}, we transport the weights for each hyperedge collectively as in \autoref{this transport}. 
Starting from the cyan hyperedge, we next transport on the orange and green hyperedges (these two are in arbitrary order), 
and finally transport the appropriate amount on the red hyperedge.

\begin{figure}[H]
\begin{tabular}{ccc}
\begin{minipage}{0.43\hsize}
\begin{center}
\begin{tikzpicture}
\node (v1) at (0,0) {};
\node (v2) at (1.2,0) {};
\node (v3) at (2.4,0) {};
\node (v4) at (0,1.2) {};
\node (v5) at (1.2,1.2) {};
\node (v6) at (2.4,1.2) {};
\node (v7) at (0,2.4) {};
\node (v8) at (1.2,2.4) {};
\node (v9) at (2.4,2.4) {};
\foreach \v in {1,2,...,9} {\fill (v\v) circle (0.1);} 
\draw (0,0)--(1.2,0);
\draw (1.2,0)--(2.4,0);
\draw (0,1.2)--(1.2,1.2);
\draw (1.2,1.2)--(2.4,1.2);
\draw (0,2.4)--(1.2,2.4);
\draw (1.2,2.4)--(2.4,2.4);
\draw (0,0)--(0,1.2);
\draw (1.2,0)--(1.2,1.2);
\draw (2.4,0)--(2.4,1.2);
\draw (0,1.2)--(0,2.4);
\draw (1.2,1.2)--(1.2,2.4);
\draw (2.4,1.2)--(2.4,2.4);
\draw (0,0)--(1.2,1.2);
\draw (1.2,0)--(2.4,1.2);
\draw (0,1.2)--(1.2,0);
\draw (1.2,1.2)--(2.4,0);
\draw (0,1.2)--(1.2,2.4);
\draw (1.2,1.2)--(2.4,2.4);
\draw (0,2.4)--(1.2,1.2);
\draw (1.2,2.4)--(2.4,1.2);
\end{tikzpicture}
\caption{The graph $H^\prime$ obtained by clique expansion of $H$ in \autoref{intro-hypergraph}.
The transport by $W_1$ follows this graph structure.} \label{clique expansion}
\end{center}
\end{minipage} 
\begin{minipage}{0.1\hsize}
\begin{center}
\end{center}
\end{minipage}
\begin{minipage}{0.43\hsize}
\begin{center}
\begin{tikzpicture}
\node (v1) at (0,0) {};
\node (v2) at (1.2,0) {};
\node (v3) at (2.4,0) {};
\node (v4) at (0,1.2) {};
\node (v5) at (1.2,1.2) {};
\node (v6) at (2.4,1.2) {};
\node (v7) at (0,2.4) {};
\node (v8) at (1.2,2.4) {};
\node (v9) at (2.4,2.4) {};
\foreach \v in {1,2,...,9} {\fill (v\v) circle (0.1);} 
\path[draw=red, dashed] ($(v4) + (-0.5,0)$) 
to[out=90,in=270] ($(v7) + (-0.5,0)$)
to[out=90,in=180] ($(v7) + (0,0.5)$)
to[out=0,in=180] ($(v8) + (0,0.5)$)
to[out=0,in=90] ($(v8) + (0.5,0)$)
to[out=270,in=90] ($(v5) + (0.5,0)$)
to[out=270,in=0] ($(v5) + (0,-0.5)$)
to[out=180,in=0] ($(v4) + (0,-0.5)$)
to[out=180,in=270] ($(v4) + (-0.5,0)$);
\path[draw=green, dashed] ($(v5) + (-0.4,0)$) 
to[out=90,in=270] ($(v8) + (-0.4,0)$)
to[out=90,in=180] ($(v8) + (0,0.4)$)
to[out=0,in=180] ($(v9) + (0,0.4)$)
to[out=0,in=90] ($(v9) + (0.4,0)$)
to[out=270,in=90] ($(v6) + (0.4,0)$)
to[out=270,in=0] ($(v6) + (0,-0.4)$)
to[out=180,in=0] ($(v5) + (0,-0.4)$)
to[out=180,in=270] ($(v5) + (-0.4,0)$);
\path[draw=orange, dashed] ($(v1) + (-0.4,0)$) 
to[out=90,in=270] ($(v4) + (-0.4,0)$)
to[out=90,in=180] ($(v4) + (0,0.4)$)
to[out=0,in=180] ($(v5) + (0,0.4)$)
to[out=0,in=90] ($(v5) + (0.4,0)$)
to[out=270,in=90] ($(v2) + (0.4,0)$)
to[out=270,in=0] ($(v2) + (0,-0.4)$)
to[out=180,in=0] ($(v1) + (0,-0.4)$)
to[out=180,in=270] ($(v1) + (-0.4,0)$);
\path[draw=cyan, dashed] ($(v2) + (-0.5,0)$) 
to[out=90,in=270] ($(v5) + (-0.5,0)$)
to[out=90,in=180] ($(v5) + (0,0.5)$)
to[out=0,in=180] ($(v6) + (0,0.5)$)
to[out=0,in=90] ($(v6) + (0.5,0)$)
to[out=270,in=90] ($(v3) + (0.5,0)$)
to[out=270,in=0] ($(v3) + (0,-0.5)$)
to[out=180,in=0] ($(v2) + (0,-0.5)$)
to[out=180,in=270] ($(v2) + (-0.5,0)$);
\draw[arrows=->, ultra thick, draw=cyan] ($(v3)+(-0.15,0.15)$) to ($(v5)+(0.15,-0.15)$);
\draw[arrows=->, ultra thick, draw=orange] ($(v2)+(-0.15,0.15)$) to ($(v4)+(0.15,-0.15)$);
\draw[arrows=->, ultra thick, draw=orange] ($(v1)+(0,0.2)$) to ($(v4)+(0,-0.2)$);
\draw[arrows=->, ultra thick, draw=green] ($(v6)+(-0.15,0.15)$) to ($(v8)+(0.15,-0.15)$);
\draw[arrows=->, ultra thick, draw=green] ($(v9)+(-0.2,0)$) to ($(v8)+(0.2,0)$);
\draw[arrows=->, ultra thick, draw=red] ($(v5)+(-0.15,0.15)$) to ($(v7)+(0.15,-0.15)$);
\draw[arrows=->, ultra thick, draw=red] ($(v4)+(0,0.2)$) to ($(v7)+(0,-0.2)$);
\draw[arrows=->, ultra thick, draw=red] ($(v8)+(-0.2,0)$) to ($(v7)+(0.2,0)$);
\node at ($(v2) + (0.2,0.2)$) {$v_1$};
\node at ($(v6) + (-0.2,-0.2)$) {$v_2$};
\end{tikzpicture}
\caption{The transport method proposed in this paper.
Our distance acts to combine the transport in the same hyperedge as possible.} \label{this transport}
\end{center}
\end{minipage}
\end{tabular}
\end{figure}

In this paper we consider a transport cost in each hyperedge to be discounted as the amount of transport increases.
We use a concave function in this discount.
Then, we take the sum of the costs of those transports in hyperedges, and introduce a new distance as the minimum of the total cost.
Although this idea of discounting the cost is a simple one, it seems to be the setting that often appears in reality.
Moreover, we stress that our transport distance is new even for graphs and seems worthy of further investigations.
Then, we generalize the coarse Ricci curvature with our distance.
As an application of this curvature, we also derived a Bonnet--Myers type estimate of the diameter of a hypergraph.

This paper is organized as follows.
In \autoref{Premininaries}, we briefly review the LLY--Ricci curvature of graphs along the line of \cite{LLY}.
\autoref{main section} is the core of this paper.
First, in \autoref{hypergraph}, we explain our setting of hypergraphs.
Then, in \autoref{SOT}, we briefly review the structured optimal transport introduced in \cite{AJJ} which is behind our construction.
Next, we describe our transport cost/distance in \autoref{idea}.
In \autoref{hcurv-subsec} we introduce the Lin--Lu--Yau type curvature associated with our transport distance for hypergraphs.
We can argue basically in the same way as \autoref{Premininaries}, 
however, we need some deformations due to the modification of the cost by a concave function.
In \autoref{Examples}, we calculate the curvature in several examples.
We also calculate the Ollivier type curvature and the LLY--Ricci curvature for comparison.
In \autoref{B-M section} we derive a Bonnet--Myers type inequality under a lower Ricci curvature bound.
Finally, we discuss some further problems related to our transport distance and its associated Ricci curvature in \autoref{FP}.

\begin{notation}
We denote the set of all positive integers by $\N$, i.e., $\N\:=\Z_{\geq1}$.
We set $[n]\:=\{1,2,\ldots,n\}$ and $[n]_0\:=[n]\cup\{0\}$ for $n\in\N$.
For $a\in\R$, denote by $\lfloor a\rfloor$ the largest integer less than or equal to $a$, and by $\lceil a\rceil$ the smallest integer greater than or equal to $a$.
For a metric space $(X,d)$, we denote the set of all Borel probability measures with bounded support on $(X,d)$ by $\PP_c(X)$.
\end{notation}

\begin{acknowledgements}
The author would like to thank his supervisor Shin-ichi Ohta for many fruitful advices, discussions and encouragements.
He is also grateful to Daisuke Kazukawa for helpful comments.
\end{acknowledgements}

\section{Premininaries} \label{Premininaries}

In this section, we review the LLY--Ricci curvature of graphs and its properties.

A \emph{graph} $G=(V,E)$ is a couple of a \emph{vertex set} $V$ and an \emph{edge set} $E$.
For vertices $x,y$, if they are connected by an edge, then we say that they are \emph{adjacent} and denote by $x\sim y$.
We set $B_1(x)\:=\big\{y\in V\;|\;d(y,x)\leq1\big\}$ and $d_x\:=\#\{y\in V \;|\; y\sim x\}$, where $d$ is 
the graph distance.
Our graphs will be connected, locally finite and without loops (see also \autoref{defs of hypers}).

\begin{defi}[Random walk with idleness parameter $\A\in\text{[0,1]}$] \label{random walk}
For each vertex $x\in V$ and $\A\in[0,1]$, we define a random walk $\rwx$ as follows: \vspace{-3mm}
\begin{center}
\[ \dis \rwx(y) \:=
\begin{cases}
\A & (y=x) , \vspace{1mm} \\
\frac{1-\A}{d_x} & (y\sim x) , \vspace{1mm} \\
0 & (y\notin B_1(x)) .
\end{cases}\]
\end{center} 
\end{defi}

The coarse Ricci curvature with idleness parameter $\A$ is defined as follows.

\begin{defi}[$\A$-Ollivier--Ricci curvature; \cite{Ol,LLY}] \label{alpha-curv}
Let $\A\in[0,1]$.
The \emph{$\A$-Ollivier--Ricci curvature $\kaxy$} between two vertices $x,y\in V$ is defined as 
\begin{align*}
\kaxy\:=1-\frac{\wxy}{d(x,y)}.
\end{align*}
\end{defi}

Note that \cite{Ol} considered the cases of $\A=0,1/2$, and the general one for $\A\in[0,1]$ was defined in \cite{LLY}.

When $\A=1$, $\K(1;x,y)=0$ for any vertices $x,y$.
Then, in \cite{LLY}, they defined a curvature as the (left) limit $\lim_{\A\uto1}\kaxy/(1-\A)$ instead of simply assigning $\A=1$ (\autoref{LLY-curv}).
To confirm the existence of this limit, the following two lemmas were shown.

\begin{lem}[{\cite[Lemma 2.1]{LLY}}] \label{global concavity}
For any $x,y\in V$, the function $\A\mapsto\K(\A;x,y)$ is concave on $[0,1]$.
\end{lem}

\begin{lem}[{\cite[Lemma 2.2]{LLY}}] \label{LLY-bdd}
For any $\A\in[0,1]$ and $x,y\in V$, we have
\begin{align*}
|\K(\A;x,y)|\leq\frac{2(1-\A)}{d(x,y)}. 
\end{align*}
\end{lem}

From \autoref{global concavity} and \autoref{LLY-bdd}, we can see the existence of the (left) limit $\lim_{\A\uto1}\K(\A;x,y)/(1-\A)$.

\begin{defi}[LLY--Ricci curvature; \cite{LLY}] \label{LLY-curv}
We define the \emph{LLY--Ricci curvature} $\kxy$ between two vertices $x,y\in V$ as
\begin{align*}
\kxy\:=\lim_{\A\uto1}\frac{\kaxy}{1-\A}. 
\end{align*}
\end{defi}

We recall basic properties of the LLY--Ricci curvature for later convenience.

\begin{prop}[{\cite[Lemma 2.3]{LLY}}] \label{LLYboundness}
If the LLY--Ricci curvature between any two adjacent vertices is greater than or equal to $\K$, then $\kxy\geq\K$ for any two vertices $x,y$.
In other words, it holds that:
\begin{align*}
\inf_{x,y\in V}\kxy = \inf_{x\sim y}\kxy.
\end{align*}
\end{prop}

A Bonnet--Myers type inequality for graphs by using the LLY--Ricci curvature was obtained as follows.

\begin{thm}[Bonnet--Myers type inequality; {\cite[Theorem 4.1]{LLY}}] \label{LLY-BM}
If there exists some positive constant $\K$ such that $\kxy\geq\K$ for any adjacent vertices $x,y$, then 
\begin{align} 
\diam(G) \leq \bigg\lfloor \frac{2}{\K} \bigg\rfloor \label{ineq:LLY-BM}
\end{align}
holds, where $\diam(G)\:=\sup_{x,y\in V}d(x,y)$.
\end{thm}

\section{$h$-LLY--Ricci curvature} \label{main section}

\subsection{Hypergraphs} \label{hypergraph}

A \emph{hypergraph} $H=(V,E)$ is a couple of a discrete set $V\neq\emptyset$ and a family of its subsets $E\neq\emptyset$.
An element $v$ of $V$ is called a \emph{vertex} and an element $e$ of $E$ is called a \emph{hyperedge}.
When $|e|=2$ for each $e\in E$, we have a usual graph.
A hyperedge $e$ with $|e|=2$ is simply called an \emph{edge}.
We may denote a hypergraph $H=(V,E)$ simply by $H$.

\begin{defi} \label{defs of hypers}
For a hypergraph $H=(V,E)$, we define the following: 
\begin{itemize}
\item Vertices $v_1,v_2,\ldots,v_k$ are \emph{adjacent} if there exists a hyperedge $e\in E$ such that $v_1,v_2,\ldots,v_k\in e$.
\item The adjacency of vertices $x,y$ is denoted by $x\sim y$.
\item Consider vertices $x,y$ joined by using vertices and hyperedges as 
\begin{align*}
x = v_0 \to e_1 \to v_1 \to e_2 \to v_2 \to \cdots \to v_{k-1} \to e_k \to v_k = y,
\end{align*}
where, for all $i\in[k]$, $v_{k-1},v_k\in e_k$.
This sequence is called a \emph{path} connecting $x,y$, 
We define the \emph{length} of a path by the number of hyperedges used in the path.
For instance, the length of the path represented above is $k$.
\item We define the \emph{graph distance} $d(x,y)$ between vertices $x,y$ by the shortest length of paths connecting $x,y$.
\item The \emph{diameter} $\diam(H)$ of the hypergraph is defined as
\begin{align*}
\diam(H)\:=\sup_{x,y\in V}d(x,y).
\end{align*}
\item We denote the \emph{sphere} $S_r(v)$ and the \emph{closed ball} $B_r(v)$ with radius $r>0$ centered at vertex $v$ by
\begin{align*}
S_r(v)\:=\{ x\in V \;|\; d(x,v)=r \},\quad B_r(v)\:=\{ x\in V \;|\; d(x,v)\leq r \}.
\end{align*}
\item The \emph{degree} $d_x$ of a vertex $x$ is defined as $d_x\:=\#\{y\in V\;|\; y\sim x\}$.
\item When $e$ is a hyperedge with $|e|=1$, $e$ is called a \emph{loop}.
When $e=\{v\}$, $v$ is said to have a loop $e$, which is regarded as $v\sim v$.
\item A hypergraph $H$ is \emph{connected} if there exists a path connecting $x,y$ for any two vertices $x,y$.
\item A hypergraph $H$ is \emph{locally finite} if $d_v<\infty$ for any vertex $v$.
\item A hypergraph $H$ is \emph{simple} if $e_1\neq e_2$ implies $e_1\not\subset e_2$ for hyperedges $e_1,e_2$ of $H$.
\end{itemize}
\end{defi}

Note that if a hypergraph is connected and simple, it has no loops unless $|V|=1$.
In this paper, our hypergraphs will be connected, locally finite and simple unless otherwise stated.

\begin{exa}
In \autoref{H_1}, let $V=\{v_i\;|\; i\in[8]\}$.
In this case, for instance, if we set $E_1=\{e_j \;|\; j\in[5]\}$, then $v_5$ has a loop $e_3$ and $H_1=(V,E_1)$ is not simple because $e_2\subset e_1$.
Also, if we set $E_2=\{e_j\;|\; j\in[4]\}$, then $H_2=(V,E_2)$ is not connected.
If we set $E_3=\{e_1,e_4,e_5\}$, then $H_3=(V,E_3)$ satisfies the requirements in this paper.

\begin{figure}[H]
\begin{center}
\begin{tikzpicture}
\node (v1) at (0,0) {};
\node (v2) at (1.5,0) {};
\node (v3) at (3,0) {};
\node (v4) at (4.5,0) {};
\node (v5) at (-1.5,-1.5) {};
\node (v6) at (0,-1.5) {};
\node (v7) at (1.5,-1.5) {};
\node (v8) at (3,-1.5) {};
\foreach \v in {1,2,...,8} {\fill (v\v) circle (0.1);} 
\path[draw=black] ($(v1) + (-1,0)$) 
to[out=90,in=180] ($(v1) + (0,0.6)$) 
to[out=0,in=180] ($(v4) + (0,0.6)$) 
to[out=0,in=90] ($(v4) + (1,0)$) 
to[out=270,in=0] ($(v4) + (0,-0.6)$) 
to[out=180,in=0] ($(v1) + (0,-0.6)$) 
to[out=180,in=270] ($(v1) + (-1,0)$);
\path[draw=black] ($(v3) + (-0.5,0)$) 
to[out=90,in=180] ($(v3) + (0,0.45)$) 
to[out=0,in=180] ($(v4) + (0,0.45)$) 
to[out=0,in=90] ($(v4) + (0.5,0)$) 
to[out=270,in=0] ($(v4) + (0,-0.45)$) 
to[out=180,in=0] ($(v3) + (0,-0.45)$) 
to[out=180,in=270] ($(v3) + (-0.5,0)$);
\path[draw=black] ($(v1) + (-0.6,0)$) 
to[out=40,in=180] ($(v1) + (0,0.4)$)
to[out=0,in=140] ($(v1) + (0.6,0)$)
to[out=315,in=130] ($(v7) + (0.55,0.05)$)
to[out=300,in=0] ($(v7) + (-0.05,-0.55)$)
to[out=180,in=0] ($(v5) + (0.05,-0.55)$)
to[out=180,in=240] ($(v5) + (-0.55,0.05)$)
to[out=50,in=225] ($(v1) + (-0.6,0)$);
\path[draw=black] ($(v5) + (-0.75,0)$) 
to[out=90,in=180] ($(v5) + (0,0.45)$)
to[out=0,in=90] ($(v5) + (0.75,0)$)
to[out=270,in=0] ($(v5) + (0,-0.45)$)
to[out=180,in=270] ($(v5) + (-0.75,0)$);
\path[draw=black] ($(v6) + (-0.5,0)$) 
to[out=90,in=180] ($(v6) + (0,0.45)$) 
to[out=0,in=180] ($(v8) + (0,0.45)$) 
to[out=0,in=90] ($(v8) + (0.5,0)$) 
to[out=270,in=0] ($(v8) + (0,-0.45)$) 
to[out=180,in=0] ($(v6) + (0,-0.45)$) 
to[out=180,in=270] ($(v6) + (-0.5,0)$); 
\foreach \v in {1,2,...,8} {\fill (v\v) circle (0.1);} 
\node at (0,-0.3) {$v_1$};
\node at (1.5,-0.3) {$v_2$};
\node at (3,-0.3) {$v_3$};
\node at (4.5,-0.3) {$v_4$};
\node at (-1.5,-1.75) {$v_5$};
\node at (0,-1.75) {$v_6$};
\node at (1.5,-1.75) {$v_7$};
\node at (3,-1.75) {$v_8$};
\node at (-1.2,0.4) {$e_1$};
\node at (3.75,0.25) {$e_2$};
\node at (-2.5,-1.5) {$e_3$};
\node at (3.75,-1.5) {$e_4$};
\node at (-1.5,-0.55) {$e_5$};
\end{tikzpicture}
\caption{A hypergraph $H_1=(V,E_1)$.} \label{H_1}
\end{center}
\end{figure}
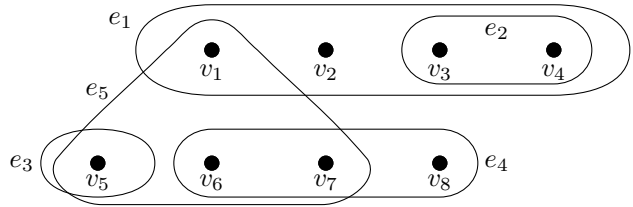
\end{exa}

\subsection{Structured optimal transport} \label{SOT}

An approach to the optimal transport problem incorporating the structure of the ground space was proposed by \cite{AJJ}.
Here we briefly explain their construction for comparison with ours in the next subsection.
They attempted to derive class-coherent matches by using submodular costs as follows.
First, we define a submodular function.

\begin{defi}[Submodular function]
Let $V$ be a discrete set and $F:\01^V\to\R$.
$F$ is said to be \emph{submodular} if 
\begin{align} 
F(S)+F(T)\geq F(S\cup T) + F(S\cap T) \label{submodular-1}
\end{align}
holds for any $S,T\subset V$.
Moreover, we say that $F$ is \emph{monotone nondecreasing} if $F(S)\leq F(T)$ holds for any $S\subset T\subset V$ 
and $F$ is \emph{normalized} if $F$ satisfies $F(\emptyset)=0$.
\end{defi}

\eqref{submodular-1} is equivalent to 
\begin{align} 
F\big(S\cup\{v\}\big)-F(S) \geq F\big(T\cup\{v\}\big)-F(T) \label{submodular-2}
\end{align}
for any $S\subset T\subset V$ and any $v\in V\backslash T$.
This indicates that a submodular function has the decreasing property.
In other words, \eqref{submodular-2} implies that the cost of an additional contract $v$ over $T$ is not higher than that over $S$.

Let $(V,d)$ be a discrete metric space and $\mu,\nu\in\PP_c(V)$.
First, we decompose $\supp\mu,\supp\nu$ as $\supp\mu=\bigsqcup_k A_k,\supp\nu=\bigsqcup_l B_l$, respectively, and set $E_{kl}$ to be the set of all matchings between $A_k,B_l$: 
\begin{align*}
E_{kl} \:= \big\{ (x,y)\in V\times V \;\big|\; x\in A_k,y\in B_l \big\}.
\end{align*}
By using a submodular function $F_{kl}:\01^{V\times V}\to\R$ with support $E_{kl}$, define $F:\01^{V\times V}\to\R$ as
\begin{align}
F(S)\:=\sum_{k,l}F_{kl}(S\cap E_{kl}) \label{F}
\end{align}
for $S\subset V\times V$.
An example of a submodular function $F_{kl}$ is a function defined as
\begin{align}
F_{kl}(S) \:= g_{kl}\left( \sum_{(x,y)\in S\cap E_{kl}} d(x,y) \right) \label{submodular constructing}
\end{align}
with a monotone nondecreasing concave function $g_{kl}:\R\to\R$.
The sum of submodular functions is also a submodular function, so $F$ in \eqref{F} is a submodular cost.
Then, \cite{AJJ} defined the following.

\begin{defi}[\cite{AJJ}] \label{submodular cost}
Let $(V,d)$ be a discrete metric space and $F:\01^{V\times V}\to[0,\infty)$ be a normalized monotone nondecreasing submodular function.
Suppose that $F$=0 on the diagonal set and $F>0$ otherwise.
Define the transport cost $d_F(\mu,\nu)$ from $\mu$ to $\nu$ as
\begin{align*}
d_F(\mu,\nu) \:= \min_{\pi\in\Pi(\mu,\nu)} f(\pi), 
\end{align*}
where $f$ is the Lov\'{a}sz extension of $F$.
\end{defi}

The submodularity of the set-function $F$ is equivalent to the convexity of its Lov\'{a}sz extension, 
thus the above definition is useful for calculation since the value of $f$ can be replaced by a maximization problem on a bounded closed set.
On the other hand, while $d_F$ is a semi-metric on $\PP(V)$ (\cite[Lemma 3.1]{AJJ}), it is not a distance function in general.
In fact, it is easy to construct examples in which the triangle inequality does not hold.
For instance, letting $V=\{v_1,v_2,v_3\}$, we can construct $F:\01^{V\times V}\to\R$ such that
\begin{align*}
d_F(\delta_{v_1},\delta_{v_3})>d_F(\delta_{v_1},\delta_{v_2})+d_F(\delta_{v_2},\delta_{v_3}).
\end{align*}
First, $F\big(\{(v_i,v_i)\}\big)=0$ holds for any $i\in[3]$ by assumption.
Moreover, denoting the $(j,k)$-matrix unit by $E_{jk}$, we have
\begin{align*}
d_F(\delta_{v_j},\delta_{v_k})=f(E_{jk})=F\big(\{(v_j,v_k)\}\big)
\end{align*}
for any $j,k\in[3]$ $(j\neq k)$ since $\Pi(\delta_{v_j},\delta_{v_k})=\{E_{jk}\}$ holds.
Therefore, if we take $F\big(\{(v_1,v_2)\}\big)=F\big(\{(v_2,v_3)\}\big)=1$ and $F\big(\{(v_1,v_3)\}\big)=2+\eps$ $(\eps>0)$, 
then we can construct a desired set-function $F$ by determining the other values of $F$ so that $F$ is a normalized monotone nondecreasing submodular function.

\subsection{New transport distance} \label{idea}

We shall introduce a new transport cost on hypergraphs utilizing a concave function in a different way from \eqref{submodular constructing}.
We remark that hyperedges are not disjoint, so that one cannot simply apply the construction in the previous subsection.
Instead, our idea is to discount the cost when the transport is done within a hyperedge.

We consider a concave function satisfying the following conditions to discount the cost.

\begin{ass} \label{h}
Assume that $h:[0,1]\to\mathbb{R}_{\geq0}$ is a monotone nondecreasing continuous concave function satisfying the following conditions:
\begin{itemize}
\item $h(0)=0$,
\item $\dis \lim_{\la\dto0} \frac{h(\la)-h(0)}{\la-0}\;\left(=\lim_{\la\dto0} \frac{h(\la)}{\la}\right)\in(0,\infty)$.
\end{itemize}
For simplicity, we set
\begin{align*} 
h^\prime(0) \:= \lim_{\la\dto0} \frac{h(\la)}{\la} ,\quad h^\prime(1) \:= \lim_{\la\uto1}\frac{h(1)-h(\la)}{1-\la}.
\end{align*}
\end{ass}

\begin{exa} \label{h-exa}
Let $a$ be a positive constant.
The linear function $h(\la)=a\la$ obviously satisfies \autoref{h}.
Other example are, e.g.,
\begin{itemize}
\item $\log$-type function: $h(\la)=a\log(1+\la)$,
\item truncation function: $h(\la)=\min\{a,\la\}$.
\end{itemize}
In the truncation function $h$, the cost is proportional to the amount of transportation until it reaches some fixed value $a$, 
and no further cost is required when it exceeds $a$.
By combining these examples, we can also consider
\begin{align*}
h(\la)=\min\{a,\la\}+\log{\big(1+|\la-a|_+\big)},
\end{align*}
where $|x|_+=\max\{x,0\}$.
\end{exa}

\begin{rem}
The power function $h(\lambda)=\lambda^a$ ($a\in(0,1)$) is concave but the limit $\lim_{\la\dto0} h(\la)/\la$ diverges.
Thus it does not satisfy \autoref{h}, though our disscusion also applies to some extent.
\end{rem}

We obtain $h^\prime(1) \leq h(1) \leq h^\prime(0)$ by the concavity of $h$.
If equality holds in one of these inequalities, then $h$ is linear.

Now we introduce a new distance between probability measures with bounded support using such a concave function $h$ and the $L^1$-Wasserstein distance $W_1$.

\begin{defi} \label{W}
Let $H=(V,E)$ be a hypergraph.
For $\mu,\nu\in\PP_c(V)$, we define $W_h:\PP_c(V)\times\PP_c(V)\to\mathbb{R}_{\geq0}$ as
\begin{align}
W_h(\mu,\nu) \:=\inf \left\{\sum\limits_{i=1}^I h \big( W_1( \xi_{i-1},\xi_i ) \big) \;\middle|\; \xi=(\xi_i)_{i\in[I]_0} \in P(\mu,\nu) \right\} , \label{Wh-def}
\end{align}
where $P(\mu,\nu)$ is the set of sequences of probability measures on $V$ with $\xi_0=\mu$ and $\xi_I=\nu$ 
such that $\supp(\xi_i-\xi_{i-1})$ is included in a hyperedge at each $i\in [I]$.
\end{defi}

By definition, $W_h(\mu,\nu)$ coincides with $W_1(\mu,\nu)$ when $h(t)=t$.

\begin{prop} \label{mini}
The minimum value on the right-hand side of \eqref{Wh-def} is actually attained.
\end{prop}

This is not straightforward since $W_h$ is defined by using the sum of the transport cost of each step.
To show \autoref{mini}, we need to have a close look on the behavior of transports with respect to $W_h$.
To this end, we introduce the following notions.

\begin{defi}
For $\mu,\nu\in\PP_c(V)$, $\xi=(\xi_i)_{i\in[I]_0}\in P(\mu,\nu)$ and $\pi\in\Pi(\mu,\nu)$, 
we define the following:
\begin{itemize}
\item $\big( \xi,(\pi_i)_{i\in[I]} \big)$ is \emph{associated} with $\pi$ 
if for all $i\in[I]$, $\pi_i\in\Pi(\xi_{i-1},\xi_i)$ and $\pi=\pi_I\pi_{I-1}\cdots\pi_1$ (i.e., $\pi$ coincides with the gluing of $\pi_1,\ldots,\pi_I$).
Moreover, we denote by $C(\pi_i) \:= \sum_{x,y\in V}d(x,y)\pi_i(x,y)$ the transport cost of $\pi_i$, and 
define $C\big((\pi_i)_{i\in[I]}\big) \:= \sum_{i\in[I]}h\big(C(\pi_i)\big)$.
\item We set $P(\mu,\nu;\pi) \:= \Big\{ \xi\in P(\mu,\nu) \;\Big|\; \big(\xi,(\pi_i)_{i\in[I]}\big)
\text{ is associated with } \pi \text{ for some } (\pi_i)_{i\in[I]} \Big\}$.
\item Let $\big( \xi,(\pi_i)_{i\in[I]} \big)$ be associated with $\pi$.
For $x,y\in V$ with $\pi\big(\{(x,y)\}\big)>0$, 
we say that $\{x_i\}_{i\in[I]_0}$ is a \emph{transport path} from $x$ to $y$ 
if $x_0=x$, $x_I=y$ and $\pi_i\big(\{(x_{i-1},x_i)\}\big)>0$ holds for any $i\in[I]$.
\end{itemize}
\end{defi}

We remark that optimal couplings are not unique for $W_1$, and hence $\pi$ is not determined from $\xi$.

\begin{lem} \label{trans path}
Given a path $\big(\xi=(\xi_i)_{i\in[I]_0},(\pi_i)_{i\in[I]}\big)$ associated with $\pi$, 
we can modify it into $\Big(\bar{\xi}=\big(\bar{\xi}_i\big)_{i\in[I]_0},(\bar{\pi}_i)_{i\in[I]}\Big)$ associated with $\pi$ 
whose transport cost is not greater than that of $\big( \xi,(\pi_i)_{i\in[I]} \big)$ such that any $x,y\in V$ with $\pi\big(\{(x,y)\}\big)>0$ have a unique transport path.
In particular, for any $i\in[I]$, it holds that:
\begin{align}
\bar{\pi}_i \big(\{(x_{i-1},x_i)\}\big) \geq \pi\big(\{(x,y)\}\big). \label{TransPathIneq}
\end{align}
\end{lem}

\begin{proof}
Suppose that $\big\{x_i^1\big\}_{i\in[I]_0}$ and $\big\{x_i^2\big\}_{i\in[I]_0}$ are 
distinct transport paths between $x,y$.
For $j=1,2$, we define 
\begin{align*}
t_j \:= \min\bigg\{ \pi_i \Big(\big\{\big(x_{i-1}^j,x_i^j\big)\big\}\Big) \;\bigg|\; 
i\in[I] \text{ such that } \big(x_{i-1}^1,x_i^1\big) \neq \big(x_{i-1}^2,x_i^2\big)\bigg\}.
\end{align*}
Now, let $s\in[-t_2,t_1]$ and for $i\in[I]$, we define $\pi_i^s\in\PP_c(V\times V)$ as
\begin{center}
\[ \dis \pi_i^s\big(\{(x_{i-1},x_i)\}\big) \:=
\begin{cases}
\pi_i\big(x_{i-1}^j,x_i^j\big)+(-1)^j s & \quad \text{if } (x_{i-1},x_i)=\big(x_{i-1}^j,x_i^j\big)\neq\big(x_{i-1}^{3-j},x_i^{3-j}\big) , \vspace{1mm} \\
\pi_i\big(x_{i-1}^j,x_i^j\big) & \quad \text{otherwise}
\end{cases}\]
\end{center} 
for $j=1,2$.
Then, $\pi_I^s\pi_{I-1}^s\cdots\pi_1^s=\pi$ holds, and note that $\pi_i^0=\pi_i$.
Let $\xi^s=(\xi_i^s)_{i\in[I]_0}$ be the path corresponding to $(\pi_i^s)_{i\in[I]}$ and observe that 
$C(\pi_i^s)$ is affine in $s$ for each $i\in[I]$.
Hence, $h\big(C(\pi_i^s)\big)$ is concave for each $i\in[I]$, and 
$C\big((\pi_i^s)_{i\in[I]}\big)=\sum_{i\in[I]}h\big(C(\pi_i^s)\big)$ is also concave.
Then $C\big((\pi_i^s)_{i\in[I]}\big)$ attains its minimum at $s=-t_2$ or $s=t_1$.
If we take $s=t_1$, $\big\{x_i^1\big\}_{i\in[I]_0}$ is no longer a transport path between $x,y$.
We define $(\bar{\pi}_i)\:=\big(\pi_i^{t_1}\big)$ and take $\bar{\xi}$ corresponding to this $(\bar{\pi}_i)$.
By iterating this procedure if necessary, the first half of the statement holds.

Furthermore, the transport path is unique for $(\bar{\pi}_i)_{i\in[I]}$ defined in this way, 
and each entry of each $\bar{\pi}_i$ is between $0$ and $1$, which implies that \eqref{TransPathIneq} holds.
\end{proof}

\begin{proof}[Proof of \autoref{mini}]
Take $\pi\in\Pi(\mu,\nu)$ and fix it.
By \autoref{trans path}, 
we can transform any $\big(\xi=(\xi_i)_{i\in[I]_0},(\pi_i)_{i\in[I]}\big)$ associated with $\pi$ into 
$\Big(\bar{\xi}=\big(\bar{\xi}_i\big)_{i\in[I]_0},(\bar{\pi}_i)_{i\in[I]}\Big)$, which has a unique transport path 
for any $x,y\in V$ with $\pi\big(\{(x,y)\}\big)>0$, without increasing the cost.
Then, since \eqref{TransPathIneq} holds, we can restrict the range of the infimum of 
\begin{align*}
W_h(\mu,\nu;\pi) &\:= \inf \left\{ \sum_{i=1}^I h\big(W_1(\xi_{i-1},\xi_i)\big) \;\middle|\; 
\xi=(\xi_i)_{i\in[I]_0}\in P(\mu,\nu;\pi) \right\} \\
&= \inf \Big\{ C\big((\pi_i)_{i\in[I]}\big) \;\Big|\; \big(\xi,(\pi)_{i\in[I]}\big) 
\text{ associated with $\pi$, having unique transport paths} \Big\}
\end{align*}
to a bounded set, and there exists a minimizer $\xi^\pi$.
Since $W_h(\mu,\nu;\pi)$ depends continuously on $\pi\in\Pi(\mu,\nu)$ and $\Pi(\mu,\nu)$ is a bounded set, $W_h(\mu,\nu)=\inf_{\pi\in\Pi(\mu,\nu)}W_h(\mu,\nu;\pi)$ holds and 
there exists $\pi$ that attains the minimum of $W_h(\mu,\nu;\pi)$.
For this $\pi$, $\xi^\pi$ constructed as above is a shortest path of $W_h(\mu,\nu)$.
\end{proof}

\begin{defi} \label{W-dist}
We call a sequence of probability measures that achieves the minimum value on the right-hand side of (\ref{Wh-def}) a \emph{shortest path of $W_h(\mu,\nu)$}.
\end{defi}
Note that, as with $W_1$, a shortest path of $W_h(\mu,\nu)$ is not uniquely determined in general.

\begin{prop}
$W_h$ defined above is a distance function on $\PP_c(V)$.
\end{prop}

\begin{proof}
We only observe the triangle inequality.
For any $\mu_1,\mu_2,\mu_3\in\PP_c(V)$, 
we obtain a path connecting $\mu_1$ and $\mu_3$ by concatenating shortest paths of $W_h(\mu_1,\mu_2)$ and $W_h(\mu_2,\mu_3)$.
Therefore, $W_h(\mu_1,\mu_3) \leq W_h(\mu_1,\mu_2) + W_h(\mu_2,\mu_3)$ holds.
\end{proof}

We emphasize that this transport distance $W_h$ is new even in the case of graphs.
In fact, by strengthening the concavity of $h$, shortest paths with respect to $W_1$ and $W_h$ may be different.

\begin{exa}
Note that in the example of \autoref{Intro}, the transport illustrated in \autoref{this transport} attains the minimum cost.
For instance, let us compare it with another transport illustrated in \autoref{AM-GM}.
We transport in the four hyperedges in the same order as \autoref{this transport} (cyan, orange, green, red), 
however, the weights on $v_1$ and $v_2$ are decomposed to make the last transport is done only for $x\to y$.
In this case, the weights transported in the cyan, orange and green hyperedges 
are all $5\B/24$, so the total cost of this transport is greater than or equal to the cost of the transport described 
in \autoref{this transport} by the concavity of $h$.
We remark that, for $W_1$, the both transports have the same costs.

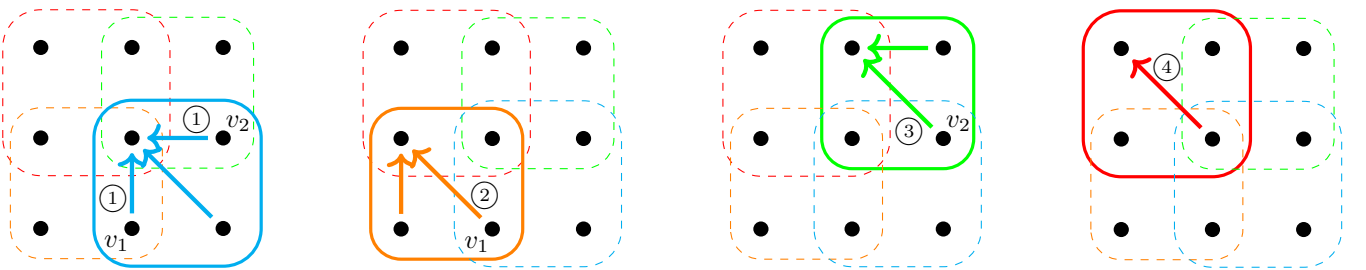
\begin{figure}[H]
\begin{tabular}{ccccccc}
\begin{minipage}{0.22\hsize}
\begin{center}
\begin{tikzpicture}
\node (v1) at (0,0) {};
\node (v2) at (1.2,0) {};
\node (v3) at (2.4,0) {};
\node (v4) at (0,1.2) {};
\node (v5) at (1.2,1.2) {};
\node (v6) at (2.4,1.2) {};
\node (v7) at (0,2.4) {};
\node (v8) at (1.2,2.4) {};
\node (v9) at (2.4,2.4) {};
\foreach \v in {1,2,...,9} {\fill (v\v) circle (0.1);} 
\path[draw=red, dashed] ($(v4) + (-0.5,0)$) 
to[out=90,in=270] ($(v7) + (-0.5,0)$)
to[out=90,in=180] ($(v7) + (0,0.5)$)
to[out=0,in=180] ($(v8) + (0,0.5)$)
to[out=0,in=90] ($(v8) + (0.5,0)$)
to[out=270,in=90] ($(v5) + (0.5,0)$)
to[out=270,in=0] ($(v5) + (0,-0.5)$)
to[out=180,in=0] ($(v4) + (0,-0.5)$)
to[out=180,in=270] ($(v4) + (-0.5,0)$);
\path[draw=green, dashed] ($(v5) + (-0.4,0)$) 
to[out=90,in=270] ($(v8) + (-0.4,0)$)
to[out=90,in=180] ($(v8) + (0,0.4)$)
to[out=0,in=180] ($(v9) + (0,0.4)$)
to[out=0,in=90] ($(v9) + (0.4,0)$)
to[out=270,in=90] ($(v6) + (0.4,0)$)
to[out=270,in=0] ($(v6) + (0,-0.4)$)
to[out=180,in=0] ($(v5) + (0,-0.4)$)
to[out=180,in=270] ($(v5) + (-0.4,0)$);
\path[draw=orange, dashed] ($(v1) + (-0.4,0)$) 
to[out=90,in=270] ($(v4) + (-0.4,0)$)
to[out=90,in=180] ($(v4) + (0,0.4)$)
to[out=0,in=180] ($(v5) + (0,0.4)$)
to[out=0,in=90] ($(v5) + (0.4,0)$)
to[out=270,in=90] ($(v2) + (0.4,0)$)
to[out=270,in=0] ($(v2) + (0,-0.4)$)
to[out=180,in=0] ($(v1) + (0,-0.4)$)
to[out=180,in=270] ($(v1) + (-0.4,0)$);
\path[draw=cyan, very thick] ($(v2) + (-0.5,0)$) 
to[out=90,in=270] ($(v5) + (-0.5,0)$)
to[out=90,in=180] ($(v5) + (0,0.5)$)
to[out=0,in=180] ($(v6) + (0,0.5)$)
to[out=0,in=90] ($(v6) + (0.5,0)$)
to[out=270,in=90] ($(v3) + (0.5,0)$)
to[out=270,in=0] ($(v3) + (0,-0.5)$)
to[out=180,in=0] ($(v2) + (0,-0.5)$)
to[out=180,in=270] ($(v2) + (-0.5,0)$);
\draw[arrows=->, ultra thick, draw=cyan] ($(v3)+(-0.15,0.15)$) to ($(v5)+(0.15,-0.15)$);
\draw[arrows=->, ultra thick, draw=cyan] ($(v2)+(0,0.2)$) to ($(v5)+(0,-0.2)$);
\draw[arrows=->, ultra thick, draw=cyan] ($(v6)+(-0.2,0)$) to ($(v5)+(0.2,0)$);
\node at ($(v2) + (-0.2,-0.2)$) {$v_1$};
\node at ($(v6) + (0.2,0.2)$) {$v_2$};
\node at ($(v2) + (-0.25,0.35)$) {$\maru{1}$};
\node at ($(v6) + (-0.35,0.2)$) {$\maru{1}$};
\end{tikzpicture}
\end{center}
\end{minipage} 
\begin{minipage}{0.03\hsize}
\begin{center}
\end{center}
\end{minipage}
\begin{minipage}{0.22\hsize}
\begin{center}
\begin{tikzpicture}
\node (v1) at (0,0) {};
\node (v2) at (1.2,0) {};
\node (v3) at (2.4,0) {};
\node (v4) at (0,1.2) {};
\node (v5) at (1.2,1.2) {};
\node (v6) at (2.4,1.2) {};
\node (v7) at (0,2.4) {};
\node (v8) at (1.2,2.4) {};
\node (v9) at (2.4,2.4) {};
\foreach \v in {1,2,...,9} {\fill (v\v) circle (0.1);} 
\path[draw=red, dashed] ($(v4) + (-0.5,0)$) 
to[out=90,in=270] ($(v7) + (-0.5,0)$)
to[out=90,in=180] ($(v7) + (0,0.5)$)
to[out=0,in=180] ($(v8) + (0,0.5)$)
to[out=0,in=90] ($(v8) + (0.5,0)$)
to[out=270,in=90] ($(v5) + (0.5,0)$)
to[out=270,in=0] ($(v5) + (0,-0.5)$)
to[out=180,in=0] ($(v4) + (0,-0.5)$)
to[out=180,in=270] ($(v4) + (-0.5,0)$);
\path[draw=green, dashed] ($(v5) + (-0.4,0)$) 
to[out=90,in=270] ($(v8) + (-0.4,0)$)
to[out=90,in=180] ($(v8) + (0,0.4)$)
to[out=0,in=180] ($(v9) + (0,0.4)$)
to[out=0,in=90] ($(v9) + (0.4,0)$)
to[out=270,in=90] ($(v6) + (0.4,0)$)
to[out=270,in=0] ($(v6) + (0,-0.4)$)
to[out=180,in=0] ($(v5) + (0,-0.4)$)
to[out=180,in=270] ($(v5) + (-0.4,0)$);
\path[draw=orange, very thick] ($(v1) + (-0.4,0)$) 
to[out=90,in=270] ($(v4) + (-0.4,0)$)
to[out=90,in=180] ($(v4) + (0,0.4)$)
to[out=0,in=180] ($(v5) + (0,0.4)$)
to[out=0,in=90] ($(v5) + (0.4,0)$)
to[out=270,in=90] ($(v2) + (0.4,0)$)
to[out=270,in=0] ($(v2) + (0,-0.4)$)
to[out=180,in=0] ($(v1) + (0,-0.4)$)
to[out=180,in=270] ($(v1) + (-0.4,0)$);
\path[draw=cyan, dashed] ($(v2) + (-0.5,0)$) 
to[out=90,in=270] ($(v5) + (-0.5,0)$)
to[out=90,in=180] ($(v5) + (0,0.5)$)
to[out=0,in=180] ($(v6) + (0,0.5)$)
to[out=0,in=90] ($(v6) + (0.5,0)$)
to[out=270,in=90] ($(v3) + (0.5,0)$)
to[out=270,in=0] ($(v3) + (0,-0.5)$)
to[out=180,in=0] ($(v2) + (0,-0.5)$)
to[out=180,in=270] ($(v2) + (-0.5,0)$);
\draw[arrows=->, ultra thick, draw=orange] ($(v2)+(-0.15,0.15)$) to ($(v4)+(0.15,-0.15)$);
\draw[arrows=->, ultra thick, draw=orange] ($(v1)+(0,0.2)$) to ($(v4)+(0,-0.2)$);
\node at ($(v2) + (-0.2,-0.2)$) {$v_1$};
\node at ($(v2) + (-0.1,0.4)$) {$\maru{2}$};
\end{tikzpicture}
\end{center}
\end{minipage}
\begin{minipage}{0.03\hsize}
\begin{center}
\end{center}
\end{minipage}
\begin{minipage}{0.22\hsize}
\begin{center}
\begin{tikzpicture}
\node (v1) at (0,0) {};
\node (v2) at (1.2,0) {};
\node (v3) at (2.4,0) {};
\node (v4) at (0,1.2) {};
\node (v5) at (1.2,1.2) {};
\node (v6) at (2.4,1.2) {};
\node (v7) at (0,2.4) {};
\node (v8) at (1.2,2.4) {};
\node (v9) at (2.4,2.4) {};
\foreach \v in {1,2,...,9} {\fill (v\v) circle (0.1);} 
\path[draw=red, dashed] ($(v4) + (-0.5,0)$) 
to[out=90,in=270] ($(v7) + (-0.5,0)$)
to[out=90,in=180] ($(v7) + (0,0.5)$)
to[out=0,in=180] ($(v8) + (0,0.5)$)
to[out=0,in=90] ($(v8) + (0.5,0)$)
to[out=270,in=90] ($(v5) + (0.5,0)$)
to[out=270,in=0] ($(v5) + (0,-0.5)$)
to[out=180,in=0] ($(v4) + (0,-0.5)$)
to[out=180,in=270] ($(v4) + (-0.5,0)$);
\path[draw=green, very thick] ($(v5) + (-0.4,0)$) 
to[out=90,in=270] ($(v8) + (-0.4,0)$)
to[out=90,in=180] ($(v8) + (0,0.4)$)
to[out=0,in=180] ($(v9) + (0,0.4)$)
to[out=0,in=90] ($(v9) + (0.4,0)$)
to[out=270,in=90] ($(v6) + (0.4,0)$)
to[out=270,in=0] ($(v6) + (0,-0.4)$)
to[out=180,in=0] ($(v5) + (0,-0.4)$)
to[out=180,in=270] ($(v5) + (-0.4,0)$);
\path[draw=orange, dashed] ($(v1) + (-0.4,0)$) 
to[out=90,in=270] ($(v4) + (-0.4,0)$)
to[out=90,in=180] ($(v4) + (0,0.4)$)
to[out=0,in=180] ($(v5) + (0,0.4)$)
to[out=0,in=90] ($(v5) + (0.4,0)$)
to[out=270,in=90] ($(v2) + (0.4,0)$)
to[out=270,in=0] ($(v2) + (0,-0.4)$)
to[out=180,in=0] ($(v1) + (0,-0.4)$)
to[out=180,in=270] ($(v1) + (-0.4,0)$);
\path[draw=cyan, dashed] ($(v2) + (-0.5,0)$) 
to[out=90,in=270] ($(v5) + (-0.5,0)$)
to[out=90,in=180] ($(v5) + (0,0.5)$)
to[out=0,in=180] ($(v6) + (0,0.5)$)
to[out=0,in=90] ($(v6) + (0.5,0)$)
to[out=270,in=90] ($(v3) + (0.5,0)$)
to[out=270,in=0] ($(v3) + (0,-0.5)$)
to[out=180,in=0] ($(v2) + (0,-0.5)$)
to[out=180,in=270] ($(v2) + (-0.5,0)$);
\draw[arrows=->, ultra thick, draw=green] ($(v6)+(-0.15,0.15)$) to ($(v8)+(0.15,-0.15)$);
\draw[arrows=->, ultra thick, draw=green] ($(v9)+(-0.2,0)$) to ($(v8)+(0.2,0)$);
\node at ($(v6) + (0.2,0.2)$) {$v_2$};
\node at ($(v6) + (-0.45,0.05)$) {$\maru{3}$};
\end{tikzpicture}
\end{center}
\end{minipage}
\begin{minipage}{0.03\hsize}
\begin{center}
\end{center}
\end{minipage}
\begin{minipage}{0.22\hsize}
\begin{center}
\begin{tikzpicture}
\node (v1) at (0,0) {};
\node (v2) at (1.2,0) {};
\node (v3) at (2.4,0) {};
\node (v4) at (0,1.2) {};
\node (v5) at (1.2,1.2) {};
\node (v6) at (2.4,1.2) {};
\node (v7) at (0,2.4) {};
\node (v8) at (1.2,2.4) {};
\node (v9) at (2.4,2.4) {};
\foreach \v in {1,2,...,9} {\fill (v\v) circle (0.1);} 
\path[draw=red, very thick] ($(v4) + (-0.5,0)$) 
to[out=90,in=270] ($(v7) + (-0.5,0)$)
to[out=90,in=180] ($(v7) + (0,0.5)$)
to[out=0,in=180] ($(v8) + (0,0.5)$)
to[out=0,in=90] ($(v8) + (0.5,0)$)
to[out=270,in=90] ($(v5) + (0.5,0)$)
to[out=270,in=0] ($(v5) + (0,-0.5)$)
to[out=180,in=0] ($(v4) + (0,-0.5)$)
to[out=180,in=270] ($(v4) + (-0.5,0)$);
\path[draw=green, dashed] ($(v5) + (-0.4,0)$) 
to[out=90,in=270] ($(v8) + (-0.4,0)$)
to[out=90,in=180] ($(v8) + (0,0.4)$)
to[out=0,in=180] ($(v9) + (0,0.4)$)
to[out=0,in=90] ($(v9) + (0.4,0)$)
to[out=270,in=90] ($(v6) + (0.4,0)$)
to[out=270,in=0] ($(v6) + (0,-0.4)$)
to[out=180,in=0] ($(v5) + (0,-0.4)$)
to[out=180,in=270] ($(v5) + (-0.4,0)$);
\path[draw=orange, dashed] ($(v1) + (-0.4,0)$) 
to[out=90,in=270] ($(v4) + (-0.4,0)$)
to[out=90,in=180] ($(v4) + (0,0.4)$)
to[out=0,in=180] ($(v5) + (0,0.4)$)
to[out=0,in=90] ($(v5) + (0.4,0)$)
to[out=270,in=90] ($(v2) + (0.4,0)$)
to[out=270,in=0] ($(v2) + (0,-0.4)$)
to[out=180,in=0] ($(v1) + (0,-0.4)$)
to[out=180,in=270] ($(v1) + (-0.4,0)$);
\path[draw=cyan, dashed] ($(v2) + (-0.5,0)$) 
to[out=90,in=270] ($(v5) + (-0.5,0)$)
to[out=90,in=180] ($(v5) + (0,0.5)$)
to[out=0,in=180] ($(v6) + (0,0.5)$)
to[out=0,in=90] ($(v6) + (0.5,0)$)
to[out=270,in=90] ($(v3) + (0.5,0)$)
to[out=270,in=0] ($(v3) + (0,-0.5)$)
to[out=180,in=0] ($(v2) + (0,-0.5)$)
to[out=180,in=270] ($(v2) + (-0.5,0)$);
\draw[arrows=->, ultra thick, draw=red] ($(v5)+(-0.15,0.15)$) to ($(v7)+(0.15,-0.15)$);
\node at ($(v7) + (0.6,-0.3)$) {$\maru{4}$};
\end{tikzpicture}
\end{center}
\end{minipage}
\end{tabular}
\caption{Another transport from \autoref{this transport}.
The weights transported at the arrows $\maru{1},\maru{2},\maru{3},\maru{4}$ are $\maru{1}=\B/24, \maru{2}=\maru{3}=\B/12, \maru{4}=\A-\B/3$.} \label{AM-GM}
\end{figure}
\end{exa}

\begin{exa} \label{shortest paths}
For a ladder graph $G$ of \autoref{graph and log}, calculate the cost of transport from $\rwx$ to $\rwy$.
In particular, we will pay attention to the way in which the weight on $v$ is transported to $\bar{v}$.
Here, let $d(x,y)\geq3$ and let $\A$ be sufficiently close to $1$.
\vspace{-5mm}
\begin{figure}[H]
\begin{tabular}{ccc}
\begin{minipage}{0.1\hsize}
\begin{center}
\end{center}
\end{minipage}
\begin{minipage}{0.75\hsize}
\begin{center}
\begin{tikzpicture}[every node/.style={circle,fill=white}]
\draw (0,0) node (x) [draw] {$x$} (1.5,0) node (a) {} (x)--(a);
\draw (1.5,0) node (a) [draw] {\;} (2.4,0) node (b) {} (a)--(b);
\draw[dotted] (2.4,0) node (b) {} (3.1,0) node (c) {} (b)--(c);
\draw (3.1,0) node (c) {} (4,0) node (d) [draw] {\;} (c)--(d); 
\draw (4,0) node (d) [draw] {\;} (5.5,0) node (y) [draw] {$y$} (d)--(y);
\draw (0,1.5) node (v1) [draw] {$v$} (1.5,1.5) node (A) {} (v1)--(A);
\draw (1.5,1.5) node (A) [draw] {\;} (2.4,1.5) node (B) {} (A)--(B);
\node at ($(B) + (0.35,0.5)$) {\small{Route $A$}};
\draw[dotted] (2.4,1.5) node (B) {} (3.1,1.5) node (C) {} (B)--(C);
\draw (3.1,1.5) node (C) {} (4,1.5) node (D) [draw] {\;} (C)--(D); 
\draw (4,1.5) node (D) [draw] {\;} (5.5,1.5) node (v2) [draw] {$\bar{v}$} (D)--(v2);
\draw (x)--(v1);
\draw (a)--(A);
\draw (d)--(D);
\draw (y)--(v2);
\draw[arrows=->, ultra thick] ($(v1)+(0.25,0.25)$) to[out=40,in=135] ($(A)+(-0.2,0.2)$);
\draw[arrows=->, ultra thick] ($(D)+(0.19,0.23)$) to[out=45,in=145] ($(v2)+(-0.3,0.25)$);
\node at (-1.2,0.75) {\small{Route $B$}};
\draw[arrows=->, draw=red, ultra thick] ($(v1)+(-0.25,-0.25)$) to[out=220,in=120] ($(x)+(-0.25,0.25)$);
\draw[arrows=->, draw=red, ultra thick] ($(y)+(0.25,0.25)$) to[out=40,in=300] ($(v2)+(0.25,-0.25)$);
\draw[arrows=->, draw=blue, ultra thick, dashed] ($(x)+(0.25,-0.35)$) to[out=0,in=180] ($(y)+(-0.25,-0.35)$);
\end{tikzpicture}
\caption{A ladder graph $G$. In the transportation of both $W_1$ and $W_h$, the weights of the lower row are transported along the lower path (a blue dotted arrow).
} \label{graph and log}
\end{center}
\end{minipage}
\begin{minipage}{0.15\hsize}
\begin{center}
\end{center}
\end{minipage}
\end{tabular}
\end{figure}
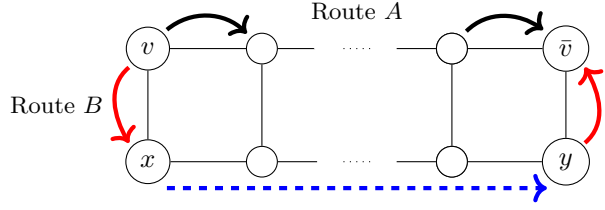 
First, the transport with respect to $W_1$ from $v$ to $\bar{v}$ passes through the upper path (black arrows in the figure).
We call this route the route $A$.
Since all the other weights are transported along the lower path, we have $\wxy=d(x,y)-(1-\A)$.
In the transport with respect to $W_h$, besides the route $A$ as in $W_1$, 
we can also transport the weight of $v$ to $x$ and then to $y$ (together with the weight on $x$) along the lower path 
and finally transport the rest to $\bar{v}$, as shown by the red arrows in the figure.
We call this route the route $B$.
The costs of the routes $A,B$ are
\begin{align*}
\text{Route $A$} &:\; d(x,y)h\left(\frac{1-\A}{2}\right)+2h(\A)+\big\{d(x,y)-2\big\}h\left(1-\frac{1-\A}{2}\right), \\
\text{Route $B$} &:\; 2h\left(\frac{1-\A}{2}\right)+2h\left(1-\frac{1-\A}{2}\right)+\big\{d(x,y)-2\big\}h(1).
\end{align*}
Thus, if the concave function $h$ satisfies
\begin{align}
2\left\{h\left(1-\frac{1-\A}{2}\right)-h(\A)\right\} \leq \big\{d(x,y)-2\big\}\left\{h\left(\frac{1-\A}{2}\right)+h\left(1-\frac{1-\A}{2}\right)-h(1)\right\}, 
\label{ladder ineq}
\end{align}
then the transport of the route $B$ is cheaper, i.e., the shortest path with respect to $W_h$ is the route $B$.
For example, the truncation function in \autoref{h-exa} satisfies \eqref{ladder ineq} for $\A$ close enough to $1$.
\qed\end{exa}

As seen in the above example, we can reduce the transport cost with respect to $W_h$ by integrating the weights into fewer points as possible.
This is a different phenomenon from $W_1$ for which one can regard that each weight is transported to its target point individually.

\subsection{$h$-LLY--Ricci curvature} \label{hcurv-subsec}

We introduce a Lin--Lu--Yau type coarse Ricci curvature by generalizing \autoref{alpha-curv} and \autoref{LLY-curv}. 
Let us define the random walk $\rwx$ starting from a vertex $x$ by the same one as \autoref{random walk} also for a hypergraph.

\begin{defi} [($\A,h$)-Ollivier--Ricci curvature] \label{alpha-curv2}
Let $\alpha\in[0,1]$.
For each pair of vertices $(x,y)$ of the hypergraph $H=(V,E)$, we define
\begin{align*}
\tkaxy \:= 1-\frac{\whxy}{\wh(\dex,\dey)}
\end{align*} 
to be the \emph{$(\A,h)$-Ollivier--Ricci curvature} $\tkaxy$ between $x,y$.
\end{defi}

\begin{rem}
Note that in \autoref{alpha-curv}, we compared $\wxy$ and $d(x,y)=\W(\dex,\dey)$.
As for $W_h$, we have $\wh(\dex,\dey)=h(1)d(x,y)$.
\end{rem}

We again observe $\tilde{\K}(1,h;x,y)=0$, and follow the procedure of \cite{LLY} we reviewed in \autoref{Premininaries}.
For this purpose, we first generalize \autoref{LLY-bdd} with a slight modification (see \autoref{on 2-boundness} after the proof).

\begin{lem} \label{boundness}
For a hypergraph $H=(V,E)$, the following inequalities hold for any $\A\in[0,1]$ and distinct $x,y\in V$.
\begin{enumerate}[(1)]
\item $\dis \tkaxy \geq -\frac{2h^\prime(0)(1-\A)}{\wh(\dex,\dey)} .$
\item $\dis \tkaxy \leq \frac{2(1-\A)}{d(x,y)} .$
\end{enumerate}
\end{lem}

\begin{proof}
$ $\newline
(1) First, from the triangle inequality, 
\begin{align} 
\whxy \leq \wh(\rwx,\dex) + \wh(\dex,\dey) + \wh(\dey,\rwy). \label{3.3.3-1}
\end{align}
The transport cost $W_h\big(\rwx,\dex\big)$ depends on how many neighbors of $x$ live a common hyperedge including $x$.
The most expensive case is that we transport from every neighbor to $x$ individually, as in a graph.
Therefore, $\wh(\rwx,\dex)$ can be bounded from above by the transportation cost in the case of a graph, namely
\begin{align*} 
\wh(\rwx,\dex)\leq d_x h\left( \frac{1-\A}{d_x} \right),\quad \wh(\rwy,\dey)\leq d_y h\left( \frac{1-\A}{d_y} \right).
\end{align*}
Combining this with the inequality (\ref{3.3.3-1}), we obtain
\begin{align*} 
\whxy
&\leq \wh(\dex,\dey) + d_x h\left( \frac{1-\A}{d_x} \right) + d_y h\left( \frac{1-\A}{d_y} \right) \leq \wh(\dex,\dey) + 2h^\prime(0)(1-\A), 
\end{align*} 
where the second inequality follows from the concavity of $h$.
This shows the claim. \\
(2) We denote by $\rwx = \mu_0 \to \mu_1 \to \cdots \to \mu_{I} = \rwy$ a shortest path of $\whxy$.
Then we have a lower bound
\begin{align} 
\whxy &= \sum_{i=1}^{I} h\left( \W(\mu_{i-1},\mu_i) \right) \geq h(1) \sum_{i=1}^{I} \W(\mu_{i-1},\mu_i) \geq h(1) \wxy, \label{W and h(1)W_1}
\end{align}
by the concavity of $h$ and the triangle inequality.
Therefore, combining this with \autoref{LLY-bdd}, we obtain 
\begin{align} 
\tkaxy &= 1- \frac{\whxy}{\wh(\dex,\dey)} \leq 1- \frac{\wxy}{d(x,y)} \leq \frac{2(1-\A)}{d(x,y)}. \label{WとW_1}
\end{align}
\end{proof}

\begin{rem} \label{on 2-boundness}
In $(2)$ above, if we begin with the triangle inequality via $\dex$ and $\dey$ similarly to \eqref{3.3.3-1}, then we have
$\whxy \geq \wh(\dex,\dey) - 2h^\prime(0)(1-\A)$ and end up with
\begin{align*}
\tkaxy \leq \frac{2h^\prime(0)(1-\A)}{\wh(\dex,\dey)} = \frac{h^\prime(0)}{h(1)}\cdot\frac{2(1-\A)}{d(x,y)},
\end{align*}
which is worse than $(2)$.
\end{rem}

Now, for taking the limit as $\A\uto1$, we would like to generalize \autoref{global concavity}.
However, due to the concavity of $h$, $\tilde{\K}(\cdot,h;x,y)$ loses the global concavity.
Toward a better understanding of the behavior of $\tilde{\K}(\cdot,h;x,y)$, the following fact in the case of graphs is worth mentioning.

\begin{thm}[{\cite[Theorem 3.4]{BCLMP}, \cite[Theorem 3.2]{CK}}] \label{3pieces}
For any connected, locally finite and simple graph $G=(V,E)$ and 
any of its vertices $x,y$, the function $\A\mapsto\kaxy$ is concave and piecewise linear on $[0,1]$.
Moreover, the number of its partitions is at most $3$.
\end{thm}

In \cite{BCLMP}, the case of $x\sim y$ was shown, 
then in \cite{CK} established the general case by using the Kantorovich duality.

More precisely, \autoref{3pieces} is shown by the property that when $\A$ moves from $0$ to $1$, the shortest path in $\wxy$ changes at most twice.
In particular, $\varphi(\A)\:=\kaxy/(1-\A)$ is constant for $\A$ sufficiently close to $1$.
On the other hand, although the existence of a shortest path for $W_h$ is shown for each $\A\in[0,1]$ as described after \autoref{W}, 
probability measures $\mu_i$ appearing in the transportation by $W_h$ is not necessarily of the type of the random walk $m_v^\A$.
Hence, it is unclear it the shortest path of $\whxy$ changes finitely many times when $\A$ moves from $0$ to $1$. 
Nonetheless, it seems plausible to expect that the following holds in the setting of our paper.

\begin{conj} \label{conj}
For any hypergraph $H=(V,E)$ and any of its vertices $x,y$, 
the function $\A\mapsto\tkaxy$ is a piecewise convex function on $[0,1]$, 
and the graph connecting the joining points with line segments gives a concave function.
\end{conj}

Then we define the LLY--Ricci curvature for $W_h$ as follows.

\begin{defi}[$h$-LLY--Ricci curvature] \label{with h}
We define the \emph{$h$-LLY--Ricci curvature} between two vertices $x,y$ as
\begin{align*} 
\tkxy \:= \liminf_{\A\uto1} \frac{\tkaxy}{1-\A} .
\end{align*}
\end{defi}

If \autoref{conj} is resolved affirmatinely, then the above $\liminf$ can be replaced with $\lim$.

We conclude this subsection by discussing fundamental properties of the $h$-LLY--Ricci curvature.
We first compare it with the LLY--Ricci curvature directly defined as an extension to hypergraphs.

\begin{prop} \label{relation}
For any hypergraph $H=(V,E)$ and any two points $x,y\in V$, we have
\begin{align}
\tkxy \leq \K(x,y). \label{3.4.1}
\end{align}
\end{prop}

\begin{proof}
By the first inequality of \eqref{WとW_1} in the proof of \autoref{boundness}, for any $\A\in[0,1]$ and $x,y\in V$, 
\begin{eqnarray*} \displaystyle
\tkaxy \leq 1-\frac{W_1(m_x^\alpha,m_y^\alpha)}{d(x,y)} = \kappa(\alpha;x,y)
\end{eqnarray*}
holds.
This implies \eqref{3.4.1}.
\end{proof}

Note that the larger the difference between $\tkxy$ and $\kxy$ is (i.e., the smaller $\tkxy$ is relative to $\kxy$), the more effective the concavity of $h$ is between $x,y$.

Finally, one can straightforwardly generalize \autoref{LLYboundness}.
Notice that the condition $\tkxy\geq\K$ is stronger than $\kxy\geq\K$ by \autoref{relation}.

\begin{prop} \label{hboundness}
If for any two adjacent vertices the $h$-LLY--Ricci curvature between them is greater than or equal to $\K$, 
then for any $x,y\in V$, $\tkxy$ is greater than or equal to $\K$.
In other words, the following holds:
\begin{align*}
\inf_{x,y\in V} \tkxy = \inf_{x\sim y} \tkxy.
\end{align*}
\end{prop}

\begin{proof}
We can argue in the same way as \cite[Lemma 2.3]{LLY} for $W_1$ thanks to the triangle inequality for $W_h$.
We remark that it is essential to define $\tkxy$ as $\liminf$ at this point.
\end{proof}

\section{Examples} \label{Examples}

In this section, we calculate $\whxy,\kaxy,\tkaxy$ for several examples, 
and compare the values of ($h$-)Ollivier--Ricci curvature for $\A=0,1/2$ and ($h$-)LLY--Ricci curvature.

\subsection{Complete graphs and cycle graphs} \label{particular}

In this subsection, we will assume that $x,y\in V$ are adjacent and $\A\in[0,1)$.
In addition, following the notation of \cite{LLY}, if a graph $G$ has the same curvature $\K$ on any edge, 
then $G$ is said to have a constant curvature $\K$ and we denote it as $\K(\A;G)=\K$ or $\tilde{\K}(h;G)=\K$.

\begin{exa}[Complete graphs $K_n$] \label{complete}
For the complete graph $K_n$, since 
$\dis \whxy =h\left(\left|\A-\frac{1-\A}{n-1}\right|\right)$, 
\begin{align*}
\tkaxy=\frac{h(1)-h\left(\left|\A-\frac{1-\A}{n-1}\right|\right)}{h(1)}
\end{align*}
holds.
Therefore, 
\begin{align*} 
\K(0;K_n)&=\frac{n-2}{n-1},& \K\left(\frac{1}{2};K_n\right)&=\frac{n}{2(n-1)},& \K(K_n)&=\frac{n}{n-1}, \\
\tilde{\K}(0,h;K_n)&=\frac{h(1)-h\big(\frac{1}{n-1}\big)}{h(1)},& \tilde{\K}\left(\frac{1}{2},h;K_n\right)&=
\frac{h(1)-h\big(\frac{n-2}{2(n-1)}\big)}{h(1)},& \tilde{\K}(h;K_n)&=\frac{n}{n-1}\frac{h^\prime(1)}{h(1)} .
\end{align*}
\qed \end{exa}

\begin{exa}[Cycle graphs $C_n$] \label{cycle}
$ $\newline
\underline{The case $n=2$}: Since $C_2=K_2$, from \autoref{complete},
\begin{align*} 
\K(0;K_2)&=0,& \K(1/2;K_2)&=1,& \K(K_2)&=2, \\
\tilde{\K}(0,h;K_2)&=0,& \tilde{\K}(1/2,h;K_2)&=1,& \tilde{\K}(h;K_2)&=\frac{2h^\prime(1)}{h(1)} .
\end{align*} \vspace{2mm}
\underline{The case $n=3,4,5$}: 
Since $\whxy=h\big(\left|\A-\frac{1-\A}{2}\right|\big)+(n-3)h\big(\frac{1-\A}{2}\big)$ holds, we have 
\begin{align*} 
\K(\A;C_n)=1-\left| \A-\frac{1-\A}{2} \right| -\frac{(n-3)(1-\A)}{2} ,\quad 
\tilde{\K}(\A,h;C_n)=\frac{h(1)-h\big(\left|\A-\frac{1-\A}{2}\right|\big)-(n-3)h\big(\frac{1-\A}{2}\big)}{h(1)}.
\end{align*} 
Hence, 
\begin{align*} 
\K(0;C_n)&=\frac{4-n}{2},& \K\left(\frac{1}{2};C_n\right)&=\frac{6-n}{4},& \K(C_n)&=\frac{6-n}{2}, \\
\tilde{\K}(0,h;C_n)&=\frac{h(1)-(n-2)h\big(\frac{1}{2}\big)}{h(1)},& 
\tilde{\K}\left(\frac{1}{2},h;C_n \right)&=\frac{h(1)-(n-2)h\big(\frac{1}{4}\big)}{h(1)},& 
\tilde{\K}(h;C_n)&=\frac{3h^\prime(1)-(n-3)h^\prime(0)}{2h(1)} .
\end{align*} \vspace{2mm}
\underline{The case $n\geq6$}: Since $\whxy=h(\A)+2h\big(\frac{1-\A}{2}\big)$ holds, we have 
\begin{align*} 
\K(\A;C_n)=0 ,\quad \tilde{\K}(\A,h;C_n)=\frac{h(1)-h(\A)-2h\big(\frac{1-\A}{2}\big)}{h(1)}.
\end{align*} 
Hence, 
\begin{align*} 
\K(0;C_n)&=0,& \K(1/2;C_n)&=0,& \K(C_n)&=0, \\
\tilde{\K}(0,h;C_n)&=\frac{h(1)-2h\big(\frac{1}{2}\big)}{h(1)},& 
\tilde{\K}\left(\frac{1}{2},h;C_n \right)&=\frac{h(1)-h\big(\frac{1}{2}\big)-2h\big(\frac{1}{4}\big)}{h(1)},& 
\tilde{\K}(h;C_n)&=\frac{h^\prime(1)-h^\prime(0)}{h(1)} .
\end{align*}
\qed \end{exa}

\subsection{Path graphs} \label{line exa}

By the definition of $W_h$, if the vertices $x,y$ are sufficiently far apart and the concavity of $h$ is strong, 
then the transportation between $\rwx$ and $\rwy$ is done along a shortest path connecting $x,y$ as seen in \autoref{shortest paths}.
We shall calculate the curvature for path graphs such as \autoref{fig8}, \autoref{fig9}, \autoref{fig10}.
The vertices $x,y$ are assumed to be at the positions specified in the respective figures, they are indeed only essential cases.

\begin{figure}[H]
\begin{center}
\begin{tikzpicture}[every node/.style={circle,fill=white}]
\draw (0,0) node (x) [draw] {$x$} (1.5,0) node (a) {} (x)--(a);
\draw (1.5,0) node (a) [draw] {\;} (2.4,0) node (b) {} (a)--(b);
\draw[dotted] (2.4,0) node (b) {} (3.1,0) node (c) {} (b)--(c);
\draw (3.1,0) node (c) {} (4,0) node (d) [draw] {\;} (c)--(d); 
\draw (4,0) node (d) [draw] {\;} (5.5,0) node (y) [draw] {$y$} (d)--(y);
\end{tikzpicture}
\caption{A path graph with $x,y$ on the ends.} \label{fig8}
\end{center}
\begin{center}
\begin{tikzpicture}[every node/.style={circle,fill=white}] 
\draw (0,0) node (x) [draw] {$x$} (1.5,0) node (a) {} (x)--(a);
\draw (1.5,0) node (a) [draw] {\;} (2.4,0) node (b) {} (a)--(b);
\draw[dotted] (2.4,0) node (b) {} (3.1,0) node (c) {} (b)--(c);
\draw (3.1,0) node (c) {} (4,0) node (d) [draw] {\;} (c)--(d); 
\draw (4,0) node (d) [draw] {\;} (5.5,0) node (y) [draw] {$y$} (d)--(y);
\draw (5.5,0) node (y) [draw] {$y$} (7,0) node (e) {} (y)--(e);
\draw (7,0) node (e) [draw] {\;};
\end{tikzpicture}
\caption{A path graph with $x$ on one end and $y$ next to the other end.} \label{fig9}
\end{center}
\begin{center}
\begin{tikzpicture}[every node/.style={circle,fill=white}]
\draw (0,0) node (f) [draw] {\;} (1.5,0) node (x) {} (f)--(x);
\draw (1.5,0) node (x) [draw] {$x$} (3,0) node (a) {} (x)--(a);
\draw (3,0) node (a) [draw] {\;} (3.9,0) node (b) {} (a)--(b);
\draw[dotted] (3.9,0) node (b) {} (4.6,0) node (c) {} (b)--(c);
\draw (4.6,0) node (c) {} (5.5,0) node (d) [draw] {\;} (c)--(d); 
\draw (5.5,0) node (d) [draw] {\;} (7,0) node (y) [draw] {$y$} (d)--(y);
\draw (7,0) node (y) [draw] {$y$} (8.5,0) node (e) {} (y)--(e);
\draw (8.5,0) node (e) [draw] {\;};
\end{tikzpicture}
\caption{A path graph with $x,y$ next to the ends.} \label{fig10}
\end{center}
\end{figure}

\begin{exa} \label{ends}
Consider the graph $G$ as in \autoref{fig8}. \\
\underline{The case $d(x,y)=1$}: 
Since $G=K_2$, the result is the same as that of \autoref{complete} and \autoref{cycle}. \vspace{2mm} \\
\underline{The case $d(x,y)\geq2$}: 
Since $\whxy=d(x,y)h(1)-2\{h(1)-h(\A)\}$, 
\begin{align*} 
\kaxy=\frac{2(1-\A)}{d(x,y)} ,\qquad \tkaxy=\frac{2\{h(1)-h(\A)\}}{d(x,y)h(1)}
\end{align*} 
hold, and thus we have 
\begin{align*} 
\K(0;x,y)&=\frac{2}{d(x,y)},& \K\left(\frac{1}{2};x,y\right)&=\frac{1}{d(x,y)},& \K(x,y)&=\frac{2}{d(x,y)}, \\
\tilde{\K}(0,h;x,y)&=\frac{2}{d(x,y)},& 
\tilde{\K}\left(\frac{1}{2},h;x,y \right)&=\frac{2\left\{h(1)-h\big(\frac{1}{2}\big)\right\}}{d(x,y)h(1)},& 
\tkxy&=\frac{2h^\prime(1)}{d(x,y)h(1)} .
\end{align*}
\qed \end{exa}

\begin{exa}
Consider the graph $G$ as in \autoref{fig9}. \\
\underline{The case $d(x,y)=1$}: 
Since $\whxy=h\big(\left|\A-\frac{1-\A}{2}\right|\big)+h\big(\frac{1-\A}{2}\big)$, 
\begin{align*}
\kaxy=1-\left|\A-\frac{1-\A}{2}\right|-\frac{1-\A}{2},\qquad 
\tkaxy=\frac{h(1)-h\big(\left|\A-\frac{1-\A}{2}\right|\big)-h\big(\frac{1-\A}{2}\big)}{h(1)}
\end{align*} 
hold, and thus we have 
\begin{align*} 
\K(0;x,y)&=0,& \K(1/2;x,y)&=1/2,& \kxy=1, \\
\tilde{\K}(0,h;x,y)&=\frac{h(1)-2h\big(\frac{1}{2}\big)}{h(1)},& 
\tilde{\K}\left(\frac{1}{2},h;x,y \right)&=\frac{h(1)-2h\big(\frac{1}{4}\big)}{h(1)},& 
\tkxy&=\frac{3h^\prime(1)-h^\prime(0)}{2h(1)} .
\end{align*} \vspace{2mm}
\underline{The case $d(x,y)\geq2$}: 
Since $\whxy=h(\A)+\big\{d(x,y)-2\big\}h(1)+h\big(\A+\frac{1-\A}{2}\big)+h\big(\frac{1-\A}{2}\big)$, 
\begin{align*} 
\kaxy=\frac{1-\A}{d(x,y)},\qquad 
\tkaxy=\frac{2h(1)-h(\A)-h\big(\A+\frac{1-\A}{2}\big)-h\big(\frac{1-\A}{2}\big)}{d(x,y)h(1)}
\end{align*} 
hold, and thus we have 
\begin{align*} 
\K(0;x,y)&=\frac{1}{d(x,y)},& \K\left(\frac{1}{2};x,y\right)&=\frac{1}{2d(x,y)},& \K(x,y)&=\frac{1}{d(x,y)}, \\
\tilde{\K}(0,h;x,y)&=\frac{2\left\{h(1)-h\big(\frac{1}{2}\big)\right\}}{d(x,y)h(1)},& 
\tilde{\K}\left(\frac{1}{2},h;x,y \right)&=
\frac{2h(1)-h\big(\frac{3}{4}\big)-h\big(\frac{1}{2}\big)-h\big(\frac{1}{4}\big)}{d(x,y)h(1)},& 
\tkxy&=\frac{3h^\prime(1)-h^\prime(0)}{2d(x,y)h(1)} .
\end{align*}
\qed \end{exa}

\begin{exa} \label{Z}
Consider the graph $G$ as in \autoref{fig10}. \\
\underline{The case $d(x,y)=1$}: 
Since $\whxy=h(\A)+2h\big(\frac{1-\A}{2}\big)$, 
\begin{align*} 
\kaxy=0,\qquad \tkaxy=\frac{h(1)-h(\A)-2h\left(\frac{1-\A}{2}\right)}{h(1)}
\end{align*} 
hold, and thus we have 
\begin{align*} 
\K(0;x,y)&=0,& \K(1/2;x,y)&=0,& \K(x,y)&=0, \\
\tilde{\K}(0,h;x,y)&=\frac{h(1)-2h\big(\frac{1}{2}\big)}{h(1)},& 
\tilde{\K}\left(\frac{1}{2},h;x,y \right)&=\frac{h(1)-h\big(\frac{1}{2}\big)-2h\big(\frac{1}{4}\big)}{h(1)},& 
\tkxy&=\frac{h^\prime(1)-h^\prime(0)}{h(1)} .
\end{align*} 
\underline{The case $d(x,y)\geq2$}: 
Since $\whxy=2h\big(\A+\frac{1-\A}{2}\big)+\big\{d(x,y)-2\big\}h(1)+2h\big(\frac{1-\A}{2}\big)$, 
\begin{align*} 
\kaxy=0,\qquad 
\tkaxy=\frac{2\left\{h(1)-h\big(\A+\frac{1-\A}{2}\big)-h\big(\frac{1-\A}{2}\big)\right\}}{d(x,y)h(1)}
\end{align*} 
hold, and thus we have 
\begin{align*} 
\K(0;x,y)&=0,& \K(1/2;x,y)&=0,& \K(x,y)&=0, \\
\tilde{\K}(0,h;x,y)&=\frac{2\left\{h(1)-2h\big(\frac{1}{2}\right)\big\}}{d(x,y)h(1)},& 
\tilde{\K}\left(\frac{1}{2},h;x,y \right)&=\frac{2\left\{h(1)-h\big(\frac{3}{4}\big)-h\big(\frac{1}{4}\big)\right\}}{d(x,y)h(1)},&
\tkxy&=\frac{h^\prime(1)-h^\prime(0)}{d(x,y)h(1)} .
\end{align*}
\qed \end{exa}

\begin{rem}
By \autoref{Z}, the LLY--Ricci curvature is $0$ for any two points in $\Z$.
In contrast, the $h$-LLY--Ricci curvature is negative if $h$ is nonlinear due to the concavity of $h$ \big($h^\prime(1) < h^\prime(0)$\big).
\end{rem}

\section{Bonnet--Myers type inequality} \label{B-M section}

We show a Bonnet--Myers type theorem for the $h$-LLY--Ricci curvature.
Recall \autoref{LLY-BM} for the Bonnet--Myers type theorem for the LLY--Ricci curvature.
We remark that if we simply follow the argument \cite{LLY}, we obtain only a weaker estimate than \eqref{ineq:LLY-BM} (see \autoref{modify}).
The following lemma is a key ingredient to improve the estimate.

\begin{lem} \label{lowerbound}
For any vertices $x,y$ such that $d(x,y)\geq2$ and any $\A\in[0,1]$, we have
\begin{align*}
\whxy\geq2h(\A)+\big\{d(x,y)-2\big\}h(1).
\end{align*}
\end{lem}

We remark that the right-hand side is the transport cost for a path graph with $x,y$ at its ends (see \autoref{ends}).
This inequality is consistent with the impression that a path graph would be most affected by the concavity of $h$.

\begin{proof}
For vertices $x,y$ such that $d(x,y)\geq2$, 
we see that $\whxy$ takes its minimum value when the hypergraph $H=(V,E)$ is a path graph $G^\prime=\big(V^\prime=\big[d(x,y)\big]_0,E ^\prime\big)$ in \autoref{Lem3.3.3}.
\begin{figure}[H]
\begin{center}
\begin{tikzpicture}[every node/.style={circle,fill=white}]
\draw (0,0) node (x) [draw] {$x$} (2,0) node (a) {} (x)--(a);
\draw (2,0) node (a) [draw] {\;} (3,0) node (b) {} (a)--(b);
\draw[dotted] (3,0) node (b) {\;} (4,0) node (c) {} (b)--(c);
\draw (4,0) node (c) {} (5,0) node (d) [draw] {\;} (c)--(d); 
\draw (5,0) node (d) [draw] {\;} (7,0) node (y) [draw] {$y$} (d)--(y);
\end{tikzpicture}
\caption{A graph in $\Z$; $G^\prime=\big([d(x,y)]_0,E^\prime\big)$, $E^\prime\:=\big\{(\ell-1,\ell) \;\big|\;\ell\in[d(x,y)]\big\}$.} \label{Lem3.3.3}
\end{center}
\end{figure} 
To this end, the vertex set $V$ of the original hypergraph $H$ is decomposed into:
\begin{align*} 
V_1 &\:= \big\{ v\in V \;\big|\; d(v,x)<d(x,y)/2 \big\} , \\
V_2 &\:= \big\{ v\in V \;\big|\; d(v,y)<d(x,y)/2 \big\} , \\
V_3 &\:= \big\{ v\in V \;\big|\; d(v,x)\geq d(x,y)/2 \text{ and } d(v,y)\geq d(x,y)/2 \big\} .
\end{align*}
Then we define a map $F:V\to \big[d(x,y)\big]_0$ by \vspace{-4mm}
\begin{center} 
\[ \dis F(v) \:=
\begin{cases} 
d(x,v) & (v\in V_1), \vspace{1mm} \\
d(x,y)-d(v,y) & (v\in V_2), \vspace{1mm} \\
\left\lfloor \frac{d(x,y)}{2} \right\rfloor & (v\in V_3).
\end{cases}\]
\end{center} 
The proof consists of two steps. \vspace{1mm} \\
\underline{\textbf{Step 1}}: We show that the function $F$ defined above is a $1$-Lipschitz function, i.e.,
\begin{align*}
d\big(F(v),F(w)\big) = |F(v)-F(w)| \leq d(v,w)
\end{align*}
holds for any $v,w\in V$.
We only need to consider the case where $v,w\in V$ are in different domains. \vspace{1mm} \\
$\maru{1}$ The case $v\in V_1,w\in V_2$:
\begin{align*} 
d\big(F(v),F(w)\big) &= \{d(x,y)-d(w,y)\}-d(x,v) \leq d(v,w) .
\end{align*}
$\maru{2}$ The case $v\in V_1,w\in V_3$:
\begin{align*} 
d\big(F(v),F(w)\big) &= \left\lfloor\frac{d(x,y)}{2}\right\rfloor -d(x,v) \leq d(x,w)-d(x,v) \leq d(v,w).
\end{align*}
The first inequality is due to the definition of $V_3$. \vspace{1mm} \\
$\maru{3}$ The case $v\in V_2,w\in V_3$:
\begin{align*} 
d\big(F(v),F(w)\big) &= \{d(x,y)-d(v,y)\}-\left\lfloor\frac{d(x,y)}{2}\right\rfloor = \left\lceil\frac{d(x,y)}{2}\right\rceil-d(v,y) \leq d(w,y)-d(v,y) \leq d(v,w).
\end{align*}
The first inequality is due to the fact that the graph distance is a non-negative integer value.

Therefore, we conclude that $F$ is a $1$-Lipschitz function. \vspace{2mm} \\
\underline{\textbf{Step 2}}: Next, we prove
\begin{align} 
\wh(F_\ast \rwx,F_\ast \rwy) \leq \whxy, \label{pushforward}
\end{align}
then the lemma is proved, where $F_\ast\rwx,F_\ast\rwy$ are the push-forward measures of $\rwx,\rwy$ by $F$, respectively.
Take a shortest path $\rwx = \mu_0 \to \mu_1 \to \cdots \to \mu_{I} = \rwy$
of the right-hand side $\whxy$ and an optimal coupling $\pi_i\in\Pi(\mu_{i-1},\mu_i)$ ($i\in[I]$).
Then, we have
\begin{align} 
\whxy &= \sum_{i=1}^I h\left( \W(\mu_{i-1},\mu_i) \right) = \sum_{i=1}^I h\left( \sum_{v,w\in V} d(v,w)\pi_i(v,w) \right) 
\geq \sum_{i=1}^I h\left( \sum_{v,w\in V} d\big(F(v),F(w)\big){\pi}_i(v,w) \right) \label{Lip-est} \\
&= \sum_{i=1}^I h\left( \sum_{a,b\in V^\prime} |a-b|{\bar{\pi}}_i(a,b) \right) \geq \wh(F_\ast \rwx,F_\ast \rwy), \nonumber 
\end{align}
which yields \eqref{pushforward}.
Here, $\bar{\pi}_i$ is the push-forward of $\pi_i$ by $F\times F$ and hence a coupling of $F_\ast(\mu_{i-1})$ and $F_\ast(\mu_i)$ for each $i\in[I]$.
Moreover, we use the $1$-Lipschitz property of $F$ shown in \textbf{Step 1} in the inequality of \eqref{Lip-est}.
\end{proof}

\begin{thm}[Bonnet--Myers type inequality] \label{BM}
For any hypergraph $H=(V,E)$, if there exists some $\K>0$ such that $\tkxy\geq\K$ for any two adjacent vertices $x,y$, then we have
\begin{align} 
\diam(H) \leq \left\lfloor \frac{h^{\prime}(1)}{h(1)}\cdot\frac{2}{\K} \right\rfloor. \label{ineq:BM}
\end{align}
\end{thm}

\begin{proof}
We first consider the case of $\diam(H)\geq2$.
Take any $x,y\in V$ and $\A\in[0,1)$ such that $d(x,y)\geq2$.
By \autoref{lowerbound},
\begin{align*} 
1-\frac{\whxy}{\wh(\dex,\dey)} 
&\leq 1-\frac{2h(\A)+\{d(x,y)-2\}h(1)}{\wh(\dex,\dey)} = \frac{2\{h(1)-h(\A)\}}{\wh(\dex,\dey)}.
\end{align*}
Dividing both sides by $1-\A$, we have 
\begin{align*} 
\frac{\tkaxy}{1-\A} \leq \frac{2}{\wh(\dex,\dey)}\cdot\frac{h(1)-h(\A)}{1-\A}.
\end{align*}
Letting $\A\uto1$ yields
\begin{align} 
\tkxy \leq \frac{2h^{\prime}(1)}{\wh(\dex,\dey)}. \label{modifiedBM}
\end{align}
Since $\tkxy>0$ and $W_h\big(\dex,\dey\big)=h(1)d(x,y)$,
\begin{align*} 
d(x,y)\leq\frac{h^{\prime}(1)}{h(1)}\cdot\frac{2}{\tkxy}. 
\end{align*}
Finally, since $\tkxy\geq\K$ from \autoref{hboundness},
\begin{align*} 
d(x,y)\leq\frac{h^{\prime}(1)}{h(1)}\cdot\frac{2}{\K}.
\end{align*} 

When $\diam(H)=1$, the distance between any two vertices is  $1$ and it is sufficient to show that 
$\tkxy \leq 2h^\prime(1)/h(1)$ for any $x,y\in V$.
Note that $\tkaxy$ is larger if and only if $\whxy$ is smaller, and 
hence the maximum value of $\tkxy$ is achieved in the complete graph $K_2$.
In $K_2$, we have $\tkxy=2h^\prime(1)/h(1)$ (see \autoref{complete}), which shows the claim.
\end{proof}

\begin{rem} \label{modify}
In \cite{LLY}, $\kaxy$ was estimated from above by using the triangle inequality via $\dex$ and $\dey$.
However, in our setup, the same argument provides only 
$\diam(H) \leq \big(h^\prime(0)/h(1)\big) \cdot (2/\K)$, 
which is weaker than the bound $\diam(H) \leq 2/\K$ in \eqref{ineq:LLY-BM} 
(even though the assumption $\tkxy \geq \K$ is stronger than $\kxy \geq \K$).
Therefore, the improvement by virtue of \autoref{lowerbound} is 
essential to capture a nontrivial result for hypergraphs.
\end{rem}

\begin{rem}
From \eqref{modifiedBM}, if $h$ satisfies $h^\prime(1)=0$, then $\inf_{x,y\in V}\tkxy\leq0$.
\end{rem}

\begin{exa}
For a complete graph $K_n$, equality holds in \eqref{ineq:BM} (\autoref{complete}).
\end{exa}

\begin{cor} \label{cpt}
A hypergraph $H$ is bounded and in particular compact if its $h$-LLY--Ricci curvature is greater than or equal to a positive constant.
\end{cor}

Moreover, the number of vertices can be estimated from above by using a lower bound of the $h$-LLY--Ricci curvature.

\begin{cor} 
For a hypergraph $H=(V,E)$, if there exists some positive constant $\K$ such that $\tkxy\geq\K$ for any two adjacent vertices $x,y$, then 
\begin{align} 
|V| \leq 1+\sum_{j=1}^{\lfloor 2h^{\prime}(1)/(h(1)\K) \rfloor}\triangle^j \prod_{i=1}^{j-1} \left( 1-\frac{\K}{2}i \right) \label{vertex-size}
\end{align}
holds.
Here, $\triangle$ is the maximum degree of the hypergraph $H$, i.e., $\triangle\:=\max_{v\in V}d_v$.
\end{cor}

\begin{proof}
As seen in \eqref{W and h(1)W_1}, $\whxy\geq h(1)\wxy$ holds for any $x,y\in V,\A\in[0,1]$.
Then we can apply the calculation in the proof of \cite[Lemma 4.4]{LLY} to $\wxy$ and find 
\begin{align*}
\tkxy \leq \frac{1}{d(x,y)} \left( 1+\frac{|\Gamma_x^-(y)|-|\Gamma_x^+(y)|}{d_y} \right),
\end{align*}
where 
\begin{align*}
\Gamma_x^+(y) \:= \big\{ v\in S_1(y) \;\big|\; d(x,v)=d(x,y)+1 \big\}, \quad 
\Gamma_x^-(y) \:= \big\{ v\in S_1(y) \;\big|\; d(x,v)=d(x,y)-1 \big\}.
\end{align*}
With this estimate in hand, we can follow the line of \cite[Theorem 4.3]{LLY} to see \eqref{vertex-size}.
\end{proof}

\section{Further Problems} \label{FP}

In this section, we discuss some problems to be considered for the further development of 
the transport distance $W_h$ and its associated Ricci curvature $\tkxy$ introduced in this paper.

\begin{duality}
As is well known, $W_1$ has a duality called the Kantorovich duality or the Kantorovich--Rubinstein formula.
Is it possible to obtain such a transformation for $W_h$?
We may restrict the class of concave functions $h$, for example, in terms of the strength of concavity.
Such a duality will increase the scope of future development, e.g., 
it would lead to the derivation of some properties such as Lipschitz contraction (\cite[Proposition 29]{Ol}).
\end{duality}

\begin{curvshape}
Recall that we defined $\tkxy$ in \autoref{with h} in the form $\liminf$, does this limit exist?
To this end, we need to examine the shape of $\tkaxy$ on $[0,1]$ as studied in \cite{BCLMP,CK} for $W_1$. 
However, as described in \autoref{conj}, the global concavity of $\tkaxy$ is lost and we expect it to be a piecewise convex function.
We need a finer analysis of transports between random walks and \autoref{trans path} would provide a clue.
\end{curvshape}

\begin{long-scale}
By the consolidating property of the transport distance $W_h$, there may be some new and interesting significance in the curvature between vertices that are not adjacent.
It would be worthwhile to consider this problem for graphs as well, since this is a unique property of $W_h$.
Furthermore, the transport path by $W_h$ between the different clusters is expected to include ``bridges'' connecting these clusters.
In this sense, $W_h$ could be useful for community detection.
\end{long-scale}

\begin{how}
Most of our estimates are in terms of $h^\prime(0)$ and $h^\prime(1)$.
It would be worthwhile considering if we can distinguish the roles of $h^\prime(0)$ and $h^\prime(1)$ 
more clearly.
In the hypothesis $h^\prime(0)\in(0,\infty)$, on one hand, $h^\prime(0)>0$ is natural for avoiding the null 
cost $h\equiv0$.
On the other hand, $h^\prime(0)<\infty$ is necessary only in the estimates using $h^\prime(0)$, 
such as some calculations of the curvature $\tkxy$ in \autoref{Examples}.
For those not including $h^\prime(0)$ like \autoref{BM}, we do not need to assume $h^\prime(0)<\infty$ and, 
e.g., the power functions $h(\lambda)=\lambda^a$ $(a\in(0,1))$ can be employed.
\end{how}

\begin{Appli}
Although the setting is different from this paper, the transport problem for a concave cost is formulated also by \cite{Xi}, and it is applied to Urban Transportation Models (see also \cite{BCM}).
It would be interesting to have a similar discussion for $W_h$ on (hyper)graphs.
In particular, the variation of shortest paths of $W_h$ in $h$ could be useful to extract new geometric features of (hyper)graphs.
For instance, the variation of the shortest paths by increasing the concavity of $h$ is similar to the distance transformation by the $d_\A$ metric in \cite{KW}, which is the distance between two vertices.
\end{Appli}


\end{document}